\documentclass[twoside,11pt]{article}
\usepackage{jmlr2e}

\usepackage{amsmath}
\usepackage{array}
\usepackage{longtable}
\usepackage{float}
\usepackage{lineno}
\usepackage{xcolor}
\usepackage{mdwlist}

\newcommand{\indep}{\rotatebox[origin=c]{90}{$\models$}}

\usepackage{lastpage}
\jmlrheading{21}{2020}{1-\pageref{LastPage}}{4/19; Revised
3/20}{7/20}{19-264}{Enguerrand Horel and Kay Giesecke}

\ShortHeadings{Significance Tests for Neural Networks}{Horel and Giesecke}
\firstpageno{1}

\begin{document}

\title{Significance Tests for Neural Networks}

\author{\name Enguerrand Horel \email enguerrand.horel@gmail.com \\
       \addr Institute for Computational and Mathematical Engineering\\
       Stanford University\\
       Stanford, CA 94305, USA
       \AND
       \name Kay Giesecke \email giesecke@stanford.edu \\
       \addr Department of Management Science and Engineering\\
       Stanford University\\
       Stanford, CA 94305, USA}

\editor{Moritz Hardt}

\maketitle

\begin{abstract}%   <- trailing '%' for backward compatibility of .sty file
%Neural networks underpin many of the best-performing AI systems. Their success is largely due to their strong approximation properties, superior predictive performance, and scalability. However, a major caveat is explainability: neural networks are often perceived as black boxes that permit little insight into how predictions are being made. We tackle this issue by 
We develop a pivotal test to assess the statistical significance of the feature variables in a single-layer feedforward neural network regression model. We propose a gradient-based test statistic and study its asymptotics using nonparametric techniques. Under technical conditions, the limiting distribution is given by a mixture of chi-square distributions. The tests enable one to discern the impact of individual variables on the prediction of a neural network. The test statistic can be used to rank variables according to their influence. Simulation results illustrate the computational efficiency and the performance of the test. An empirical application to house price valuation highlights the behavior of the test using actual data. 
\end{abstract}


\section{Introduction}
\label{S:1}

Neural networks underpin many of the best-performing artificial-intelligence systems, including speech recognizers on smartphones or Google’s latest automatic translator. The tremendous success of these applications has spurred the interest in applying neural networks in a variety of other fields, including medicine, physics, economics, marketing, and finance. However, the difficulty of interpreting a neural network has slowed the practical implementation of neural networks in these fields, where the major stakeholders often insist on the explainability of predictive models.

%finance and economics, for instance, researchers have developed several high-impact applications in risk management, asset pricing, and investment management. A first wave of work includes \cite{kuan1994artificial}, \cite{lee1993testing}, \cite{swanson1997model}, \cite{brown1998dow}, \cite{bansal1993no}, \cite{hutchinson1994nonparametric}, \cite{desai1996comparison}.  More recent work includes \cite{sirignano2016deep}, \cite{sirignano2018deep}, \cite{sirignano2018universal}, \cite{heaton2017deep}, \cite{gu2018empirical}, \cite{teichmann2018deep}, \cite{becker2018deep}, \cite{chen2018deep}.  However, the difficulty of interpreting a neural network has slowed the implementation of these applications in financial practice, where regulators, investors and other stakeholders insist on model explainability. 

This paper approaches the neural network explainability problem from the perspective of statistical significance testing. We develop and analyze a pivotal test to assess the statistical significance of the feature variables of a single-layer, fully connected feedforward neural network that models an unknown regression function.\footnote{We consider a setting with structured data but the approach presented here could also be potentially used with images or text data.} We construct a gradient-based test statistic that represents a weighted average of the squared partial derivative of the neural network estimator with respect to a variable. The weights are prescribed by a positive measure that can be chosen freely. We study the large-sample asymptotic behavior of the test statistic. The neural network estimator is regarded as a nonparametric sieve estimator.\footnote{See \cite{chen2007large} for an overview.}  The dimension of the network (i.e., the number of hidden units) grows with the number of samples. Using an empirical-process approach, we first show that the large-sample asymptotic distribution of the rescaled neural network sieve estimator is given by the argmax of a Gaussian process. The Gaussian process is indexed by the function space that contains the unknown regression function. The rescaling is given by the estimation rate of the neural network identified by \cite{chen1998sieve}. A second-order functional delta method is then used to obtain the asymptotic distribution of the test statistic as the weighted average of the squared partial derivative of the argmax of the Gaussian process with respect to the variable of interest. The empirical test statistic, which results from a weight measure equal to the empirical law of the feature variables, has the same asymptotics. Moreover, under mild assumptions, the asymptotic distribution can be represented by a mixture of chi-square distributions. We develop several approaches to compute the limiting law; these and the performance of the test are illustrated in a simulation study. An empirical application to house price valuation in the United States complements our theoretical results and highlights the behavior of the test using actual data. 

The significance test we propose has several desirable characteristics. First, the test procedure is computationally efficient. The partial derivative underpinning the test statistic is basically a byproduct of the widely used gradient-based fitting algorithms and is provided by standard software packages used for fitting neural networks (e.g., TensorFlow). Furthermore, the test procedure does not require re-fitting the neural network. Second, the test statistic can be used to rank order the variables according to their influence on the regression outcome. Third, the test is not susceptible to the non-identifiability of neural networks.\footnote{See, for example,  \cite{chen1993geometry} for a discussion of identifiability.}  Forth, the test facilitates model-free inference: in the large-sample asymptotic regime, due to the universal approximation property of neural networks (\citealt{hornik1989multilayer}), we are  performing inference on the ``ground-truth" regression function. Finally, the test statistic can be extended to higher-order and cross derivatives, facilitating tests of the statistical significance of nonlinear features such as interactions between variables.  

The rest of the paper is organized as follows. The remainder of the introduction discusses the related literature. Section \ref{S:2} discusses the model, hypotheses and test statistic. Section \ref{S:3} analyzes the large-sample asymptotic distribution of the test statistic. Section \ref{S:4} provides several approaches to compute the limiting distribution. Section \ref{S:5} analyzes the performance of the test in a simulation study. Section \ref{S:6} develops an empirical application to house prices valuation. Section \ref{S:7} concludes. There are several appendices, one containing the proofs of our technical results.

\subsection{Related Literature}
There is a growing literature on explaining black-box machine learning models using feature importance analysis (also known as saliency mask or sensitivity analysis). \cite{guidotti2018survey} provides a survey.  Several papers develop procedures to explain the prediction of a model {\it locally} at a specific instance. \cite{shrikumar2017learning} decomposes the neural network prediction for a specific input by backpropagating the contributions of all neurons in the network to every input feature. \cite{kononenko2010efficient} explains individual predictions of classification models using the game theoretic concept of Shapley value.  \cite{ribeiro2016should} explains the predictions of a classifier by learning an interpretable model locally around the prediction. \cite{baehrens2010explain} provides local explanation vectors for classification models by computing class probability gradients. These methods are limited to local explanations and cannot provide a global understanding of model behavior. 

In contrast, a second group of methods aim for understanding how the model works overall. \cite{datta2016algorithmic} measures the influence of an input on a quantity of interest by computing the difference in the quantity of interest when the data are generated according to the real distribution and when they are generated according to a hypothetical distribution designed to disentangle the effect of different inputs. \cite{sundararajan2017axiomatic} attributes the prediction of a deep network to its input features by integrating gradients of the model along the straightline path from a baseline value of the input to its current value. \cite{cortez2011opening} assesses the importance of inputs through sensitivity analysis. While similar to our approach in terms of using network gradients, these methods focus on feature importance and do not treat statistical significance. We, in contrast, develop a procedure that establishes the statistical significance of input features. Once significance has been established, the magnitude of the test statistic can be used to judge relative feature importance for the set of significant features. This can then facilitate feature selection.

There is prior work treating the significance of variables in neural network regression models. Regarding neural networks as parametric formulations, \cite{olden2002illuminating} propose a randomization test for the significance of the model's parameters. This is combined with the weight-based metric of feature importance introduced by \cite{Garson1991} to test for the statistical significance of input variables. However, this metric is not necessarily identifiable due to the non-identifiability of neural networks.  %Our approach to variable significance and importance is not susceptible to the non-identifiability.}

The large-sample asymptotic results for the estimators of neural network parameters developed by \cite{white1989learning} and \cite{white1989some} can be used to develop significance tests for the network weights and the variables. \cite{refenes1999neural} construct a test statistic for variable significance whose distribution is estimated by sampling from the asymptotic distribution of the parametric neural network estimator.  \cite{white2001statistical} develop expressions for the asymptotic distributions of two test statistics for variable significance. A bootstrap procedure is used to estimate these distributions.

Likelihood ratio tests could also be used to assess variable significance in a parametric setting. One would fit an unconstrained model incorporating all variables and nested ones with restricted variables sets. If the unconstrained model is assumed to be correctly specified, then we have standard chi-square asymptotics. Theorem 7.2 from \cite{vuong1989likelihood} provides the asymptotic distribution under mis-specification, a weighted sum of chi-square distributions parametrized by the eigenvalues of a $p\times p$ matrix where $p$ is the number of parameters of the network. While appealing, this approach tends to be computationally challenging because $p$ is typically large and one needs to re-fit the neural network multiple times.

Alternatively, neural networks can be regarded as nonparametric models, and this is the approach we pursue. This approach has several advantages over the parametric setting considered by the aforementioned authors. First, viewing a neural network as a mapping between input variables and regression output, identifiability is no longer an issue. Moreover, by not restricting the dimension of the network (i.e., the number of hidden nodes) a priori, we can exploit the consistency of neural networks (\citealt{white1990connectionist}) which is due to their universal approximation property. Model misspecification is no longer an issue. 

Nonparametric variable significance tests fall in two categories: goodness-of-fit tests or conditional moment tests; see Section 6.3 of \cite{henderson2015applied} for a review. Goodness-of-fit tests are based on the residuals from an unrestricted estimator fitted using all variables and the residuals from a restricted estimator fitted using all variables except the one being tested. The required distributions are either estimated using the bootstrap, or approximated asymptotically for the specific case of kernel regressions (see, for example, \citealt{gozalo1993consistent}, \citealt{lavergne1996nonparametric}, \citealt{yatchew1992nonparametric} and \citealt{fan2005nonparametric}). Because the tests require the estimation of a potentially large number of models, they can be computationally challenging, especially if employing the bootstrap. Conditional moment tests rely on the observation that the expectation of the residual of the restricted estimator conditional on all variables is zero under the null hypothesis of insignificance of the variables being tested. \cite{fan1996consistent}, \cite{lavergne2000nonparametric}, \cite{gu2007bootstrap}, and others study these tests in the context of kernel regressions. 

\cite{racine1997consistent} proposes, again for kernel regressions, a nonparametric variable significance test based on the partial derivative of the fitted regression function with respect to a variable. The distribution of the test statistic is estimated using the bootstrap. Our test is also based on partial derivatives but we analyze a neural network regression rather than a kernel regression. Unlike \cite{racine1997consistent}, we study the asymptotic distribution of the test statistic in order to avoid the use of the bootstrap, which tends to be computationally expensive in the context of neural network models. The large data sets that are often used in practice to estimate neural network regressions justify the use of the asymptotic distribution for the purpose of significance testing.

\section{Model, Hypotheses, and Test Statistic}
\label{S:2}

Fix a probability space $(\Omega, \mathcal{F}, \mathbb{P})$ and consider the following regression model:
\begin{equation}
\label{regression setting}
    Y = f_0(X) + \epsilon.
\end{equation}
Here, $Y \in \mathbb{R}$ is the dependent variable, $X \in \mathcal{X} \subset \mathbb{R}^d$ is a vector of regressor or feature variables with distribution $P$ for some $d \ge 3$, and $f_0: \mathcal{X} \to \mathbb{R}$ is an unknown deterministic regression function that is continuously differentiable. The error variable $\epsilon$ satisfies the usual assumptions: $\epsilon\indep X$, $\mathbb{E}(\epsilon)= 0$, and $\mathbb{E}(\epsilon^2)= \sigma^2<\infty$.
% \begin{align}
 %\label{regression error}
 %\epsilon &\indep X,  & \mathbb{E}(\epsilon)&= 0, & \mathbb{E}(\epsilon^2)&= \sigma^2.
 %\end{align}

We are interested in assessing the influence of a variable $X_j$ on $Y$. To this end, it is natural to consider the partial derivative of $f_0(x)$ with respect to $x_j$. If $\frac{\partial f_0(x)}{\partial x_j} = 0$ for all $x \in \mathcal{X}$, then the $j$th variable does not have an impact on $Y$. The significance test we propose is based on the weighted average of the partial derivative, with weights defined by a positive measure $\mu$. Specifically, we will consider the following hypotheses
%We are interested in testing the hypothesis of insignificance of input variable $j$ as defined by the null and alternative hypotheses:
%\begin{equation}
 %   \label{null hypothesis}
%    H_0: \forall x \in \mathcal{X}, \frac{\partial f_0(x)}{\partial x_j} = 0 
%\end{equation}
%\begin{equation}
%    \label{alternative hypothesis}
%    H_A: \exists x \in \mathcal{X}, \frac{\partial f_0(x)}{\partial x_j} \neq 0 
%\end{equation}
%In order to allow for a tractable testing procedure and for a looser definition of insignificance, we consider an integrated version of the squared partial derivative of $f_0$ with respect to a weight measure $\mu$. This leads to the following definitions of the test hypotheses:
\begin{align}\label{null hypothesis 2}
    H_0: \lambda_j & = \int_{\mathcal{X}} \Big(\frac{\partial f_0(x)}{\partial x_j}\Big)^2d\mu(x) = 0,\\ 
    H_A: \lambda_j & \neq 0. %\label{alternative hypothesis 2}
\end{align}
We take the square of the partial derivative to avoid compensation of positive and negatives values, ensure differentiability, and help discriminate between large and small values.  Examples of the weight measure $\mu$ include uniform weights, $d\mu(x) = dx$, uniform weights over subsets $I \subset \mathcal{X}$, $d\mu(x) = 1_{\{x \in I\}}dx$, and the choice $\mu = P$, where $P$ is the distribution of $X$. Under the latter choice, $\lambda_j$ takes the form $\lambda_j = \mathbb{E}[(\frac{\partial f_0(x)}{\partial x_j})^2]$.
%    \begin{equation}
 %       \lambda_j = \mathbb{E}\Big[\Big(\frac{\partial f_0(x)}{\partial x_j}\Big)^2\Big]
  %  \end{equation} 
%can be chosen by the user depending on the application at stake and give the flexibility to define different measure of importance and focus on different subset of the input space $\mathcal{X}$. Typical choices could be:
%\begin{itemize}
%    \item $d\mu(x) = dx$, uniform integral of the square of the partial derivative over the input domain,
%    \item $d\mu(x) = 1_{\{x \in I\}}dx$ if we are interested in measuring the effect of the variable over the sub-region $I \subset \mathcal{X}$.
%    \item $\mu = P$, where $P$ is the distribution of the input variables $X$, which define $\lambda_j$ as the square of the partial derivative over the input domain
%    \begin{equation}
 %       \lambda_j = \mathbb{E}_X\Bigg[\Big(\frac{\partial f_0(X)}{\partial x_j}\Big)^2\Bigg]
  %  \end{equation}
%\end{itemize}

If the regression function $f_0$ is assumed to be linear, i.e., $f_0(x)=\sum_{k=1}^d \beta_kx_k$, then $\lambda_j\propto \beta_j^2$ and the null takes the form $H_0: \beta_j=0$. This hypothesis can be tested using a standard $t$-test. In the general non-linear case, the derivative $\frac{\partial f_0(x)}{\partial x_j}$ is not flat but depends on $x$. To construct the null in this case, we take a weighted average of the squared values of the derivative, with $\mu$ prescribing the weights over the feature space $\mathcal{X}$.  

%In practice, the regression function $f_0$ is unknown and must be estimated. 
We study the case where the regression function $f_0$ is modeled by a fully-connected, single-layer feed-forward neural network $f$, which is specified by a bounded {activation function} $\psi$ on $\mathbb{R}$ and the number of {hidden units} $K$,
\begin{equation}\label{NN}
f(x) = b_0 + \sum_{k=1}^{K} b_k \psi(a_{0,k} + a_k^\top x),
\end{equation}
where $b_0, b_k, a_{0,k} \in \mathbb{R}$ and $a_k \in \mathbb{R}^d$ are parameters to be estimated. Functions of the form (\ref{NN}) are dense in $C(\mathcal{X})$. That is, they are {universal approximators}: choosing the dimension $K$ of the network large enough, $f$ can approximate $f_0$ to any given precision.

%Then given $n$ i.i.d. observations generated  from the regression setting \eqref{regression setting}, $(Y_i, X_i)_{i=1}^n$, one can estimate the unknown function $f_0$ by fitting a neural network estimator $f_n$ to these observations. More specifically, we consider the class of neural networks with one hidden-layer and sigmoid activation function: 
%\begin{equation}
%\label{NN space}
%\begin{split}
%       \Bigg\{\theta \in \Theta: \theta(z) = b_0 + \sum_{j=1}^{v} b_j \psi(a_j^Tz + a_{0,j}), \\ \sum_{j=0}^{v} |b_j| \leq c, \max_{1 \leq j \leq v_n}\sum_{i=0}^d |a_{i,j}|\leq \tilde{c}\Bigg\}
%\end{split}
%\end{equation}
%where $\psi$ is a Lipschitz sigmoid function, i.e. a bounded measurable function on $\mathbb{R}$ such that $\psi(x) \to 1$ as $x \to \infty$ and $\psi(x) \to 0$ as $x \to -\infty$. The vectors $a_j \in \mathbb{R}^d$ contains the weights of the edges connecting the input variables to the hidden nodes, $a_0 \in \mathbb{R}$ is weight to the bias node, the parameters $b_j \in \mathbb{R}$ are the weights connecting the hidden nodes to the output of the network and $b_0$ is the intercept of the model. 

Let $f_n$ be an estimator of $f$ based on $n$ i.i.d. samples of $(Y,X)$. The neural network test statistic $\lambda_j^n$ associated with (\ref{null hypothesis 2}) takes the form
\begin{equation}
\label{test statistic}
    \lambda_j^n = \int_{\mathcal{X}} \Big(\frac{\partial f_n(x)}{\partial x_j}\Big)^2d\mu(x).
\end{equation}
Below, we are going to study the asymptotic distribution of $\lambda_j^n$ as $n\to\infty$ under the assumption that the dimension $K=K_n$ of the neural network grows with $n$. The asymptotic distribution will be used to construct a test for the null in (\ref{null hypothesis 2}). This approach is typically less expensive computationally than a bootstrap approach. Note that in the asymptotic regime, due to the universal approximation property, we are actually performing inference on the ``ground truth" $f_0$ (i.e., model-free inference).

\section{Asymptotic Analysis}
\label{S:3}

The neural network test statistic takes the form $\lambda_j^n = \phi[f_n]$, where $\phi$ is the functional
\begin{equation}
\label{functional def}
    \phi[h] = \int_{\mathcal{X}} \Big(\frac{\partial h(x)}{\partial x_j}\Big)^2d\mu(x).
\end{equation}
Thus, we first study the asymptotic distribution of $f_n$, and then infer the asymptotic distribution of $\lambda_j^n$ under $H_0$ from that of $f_n$ using a functional delta method.

%Hence, the distribution of the statistic is derived in two steps: first the asymptotic distribution in a random function sense of the neural network is derived using the theory of empirical processes for sieved M-estimators in \ref{SS:3.1}, then the distribution of the corresponding distribution of the statistic under the null is obtained from this using a second order functional delta method in \ref{SS:3.2} and \ref{SS:3.3}. 

\subsection{Asymptotics of Neural Network Estimator}
\label{SS:3.1}
We consider $f_n$ as a nonparametric estimator. This estimator asymptotically approximates the unknown function $f_0$ which is assumed to belong to an infinite-dimensional function space $\Theta$. Since infinite-dimensional estimation does not necessarily lead to a well-posed optimization problem, we construct $f_n$ as a sieve estimator. Intuitively, the method of sieves entails the optimization of an empirical criterion over a sequence of approximating spaces called sieves which are less complex than but dense in the original infinite-dimensional space. This approach renders the resulting optimization problem well-posed. More details can be found in \cite{chen2007large}.

The unknown regression function $f_0$ is assumed to belong to the function space
\begin{equation}
\label{def function space}
    \Theta = \Big\{f \in \mathcal{C}^1, f:\mathcal{X} \subset \mathbb{R}^d \to \mathbb{R}, ||f||_{\lfloor \frac{d}{2}\rfloor+2} \leq B\Big\},
\end{equation}
where $B$ is a constant, the feature space $\mathcal{X}$ is a hypercube of dimension $d$ and  
\begin{equation}
\label{definition sobolev norm}
    ||f||_{k} = \sqrt{\sum_{0 \leq |\alpha| \leq k} \mathbb{E}_{X}\big( \nabla^{\alpha}f(X)^2\big)}. 
\end{equation}

We construct the neural network estimator $f_n$ as a sieved M-estimator. Given $n$ i.i.d. samples $(Y_i,X_i)_{i=1}^n$ of the output-input pair $(Y,X)$ and a nested sequence of sieve subsets $\Theta_n \subseteq \Theta$ such that $\cup_{n \in \mathbb{N}} \Theta_n$ is dense in $\Theta$, the neural network estimator $f_n$ is defined as the approximate maximizer of the empirical criterion function 
\begin{equation}
\label{empirical criterion}
   L_n(g) = \frac{1}{n}\sum_{i=1}^n l(Y_i,X_i,g), \quad g \in \Theta,
\end{equation}
with loss function $l: \mathbb{R}\times\mathcal{X}\times\Theta \to \mathbb{R}$ over the sieve space $\Theta_n$, i.e., 
\begin{equation}
\label{estimator definition}
    L_n(f_n) \geq \sup_{g \in \Theta_n} L_n(g) - o_P(1).
\end{equation}
%where
%\blue{
%\begin{equation}
%\label{estimation rate}
%r_n = \Big(\frac{n}{\log n}\Big)^{\frac{d+1}{2(2d+1)}}.
%\end{equation}}
% In \cite{barron1993universal}, approximation rates of neural networks have been derived for target functions $f \in \Theta$ that satisfy the following condition on their Fourier transform $\tilde{f}$,  
% \begin{equation}
% \label{first function space}
%     \Theta = \{f: \mathbb{R}^d \to \mathbb{R}, \int_{\mathbb{R}^d} |\omega|_1|\tilde{f}(\omega)|d\omega < \infty \}
% \end{equation}
% which may be interpreted as the integrability of the Fourier transform of the gradient of the function $f$.
In our setting, a sieve subset $\Theta_n$ is generated by a neural network whose dimension $K=K_n$ depends on $n$ and whose activation function $\psi$ is a Lipschitz sigmoid function, i.e., a bounded measurable function on $\mathbb{R}$ such that $\psi(x) \to 1$ as $x \to \infty$ and $\psi(x) \to 0$ as $x \to -\infty$, 
\begin{equation}
\label{sieve space}
       \Theta_n = \Bigg\{f \in \Theta: f(x) = b_0 + \sum_{k=1}^{K_n} b_j \psi(a_{0,k} + a_k^\top x),\sum_{k=0}^{K_n} |b_k| \leq c, \max_{1 \leq k \leq K_n}\sum_{i=0}^d |a_{i,k}|\leq \tilde{c}\Bigg\}.
\end{equation}
%This is the set of single-layer neural networks of dimension $K_n$ with sigmoidal activation function and bounded $L^1$ norm of weights. 
%The number of hidden units $K_n$ goes to infinity as $n \to \infty$. 
The upper bound on the norm of the weights $b_j$'s allows us to view $\Theta_n$ as the symmetric convex hull of the class of sigmoid functions $\{\psi(a^\top \cdot + a_0), a \in \mathbb{R}^d, a_0 \in \mathbb{R}\}$. It hence ensures that the integral of the metric entropy of this class of functions is bounded for every $n$ due to Theorem 2.6.9 of \cite{van1996weak}. This property is necessary for Theorem \ref{theorem: NN distribution} below and the reason why we restrict our analysis to a single-hidden layer network with sigmoid activation function. Indeed, the most recent bounds on the metric entropy of more general deep neural networks (\citealt{barron2019complexity}) violate this property.
\cite{cybenko1989approximation} proves the denseness of $\cup_{n \in \mathbb{N}} \Theta_n$ in $\Theta$.

\begin{theorem}[Asymptotic distribution of neural network estimator]
\label{theorem: NN distribution}

%Given $n$ i.i.d. observations $(Y_i,X_i)_{i=1}^n$ generated from the regression setting described in \eqref{regression setting} and \eqref{regression error} and under the following assumptions:
Assume that
\begin{enumerate}
    \item The law $P$ of the feature vector $X$ is absolutely continuous with respect to Lebesgue measure with bounded and strictly positive density;
%    \item $f_0$ belongs to the metric space $(\Theta,d)$ where $\Theta$ is defined in \eqref{def function space} and $d(f,g) = \mathbb{E}[(f-g)^2]$,
    \item The loss function $l(y, x,g) = -\frac{1}{2}(y - g(x))^2$;
    \item The dimension $K_n$ of the neural network satisfies $K_n^{2+1/d}\log K_n = O(n)$;
%\suspend{enumerate}
%Define $l_f(X,\epsilon) = 2(f-f_0)(X)\epsilon - (f_0-f)^2(X)$ where $f_0$ is the unknown regression function and $\epsilon$ the regression error in (\ref{regression setting}), and further assume that
%\resume{enumerate}
\item Letting $l_f(X,\epsilon) = 2(f-f_0)(X)\epsilon - (f_0-f)^2(X)$ where $f_0$ is the regression function and $\epsilon$ the regression error in \eqref{regression setting}, it holds that
\begin{multline*}
\frac{r_n}{\sqrt{n}} \sum_{i=1}^n \Big(l_{f_n}(X_i, \epsilon_i) -\mathbb{E}[l_{f_n}(X,\epsilon)] \Big) 
\\ \geq  \sup_{h \in \Theta} \frac{r_n}{\sqrt{n}} \sum_{i=1}^n \Big(l_{f_0 + \frac{h}{r_n}}(X_i, \epsilon_i) - \mathbb{E}[l_{f_0 + \frac{h}{r_n}}(X,\epsilon)] \Big)  -o_P(1)
\end{multline*}
where for all $n \geq 2$
\begin{equation}
\label{estimation rate}
    r_n = \Big(\frac{n}{\log n}\Big)^{\frac{d+1}{2(2d+1)}}.
\end{equation}
\end{enumerate}
Then $f_n$ converges weakly in the metric space $(\Theta,d)$ with $d(f,g) = \mathbb{E}[(f-g)^2]$:
\begin{equation}
\label{distribution of NN estimator}
    r_n(f_n - f_0) 	\Longrightarrow h^{\star},
\end{equation}
%with 
%where $\Longrightarrow$ denotes weak convergence in the metric space $(\Theta,d)$ with $d(f,g) = \mathbb{E}[(f-g)^2]$, and 
where $h^{\star}$ is the argmax of the Gaussian process $\{\mathbb{G}_f: f \in \Theta\}$ with mean zero and covariance function $Cov(\mathbb{G}_s,\mathbb{G}_t) = 4\sigma^2\mathbb{E}_{X}(s(X)t(X))$. 
\end{theorem}

%The following result develops the asymptotic properties of the neural network sieve estimator $f_n$. 
Theorem \ref{theorem: NN distribution} derives the asymptotic distribution of the neural network estimator in a random function sense using the theory of empirical processes (\citealt{van1996weak}). It complements several other convergence results for neural network estimators, including the consistency result of \cite{white1990connectionist} and results on the estimation rate and asymptotic normality of functionals of neural network estimators in \cite{shen1994convergence}, \cite{shen1997methods}, \cite{chen1998sieve} and \cite{chen1999improved}.  

\subsection{Asymptotic Distribution of Test Statistic}
\label{SS:3.2}

Using Theorem \ref{theorem: NN distribution}, we can now derive the distribution of the test statistic under the null hypothesis using a functional delta method. A second order approach is needed since a first-order method would lead to a degenerate limiting distribution under the null. 

\begin{theorem}[Asymptotic distribution of the test statistic]
\label{theorem: test statistic distribution}
Under the conditions of Theorem \ref{theorem: NN distribution} and the null hypothesis $H_0$, we obtain:
\begin{equation}\label{limit}
    r_n^2 \lambda_j^n %= r_n^2 \int_{\mathcal{X}} \Big(\frac{\partial f_n(x)}{\partial x_j}\Big)^2d\mu(x) 
    \Longrightarrow \int_{\mathcal{X}}\Big(\frac{\partial h^{\star}(x)}{\partial x_j}\Big)^2d\mu(x).
\end{equation}
%where $h^{\star}$ is the argmax the Gaussian process $\{\mathbb{G}_f: f \in \Theta\}$ with zero mean and covariance function $Cov(\mathbb{G}_s,\mathbb{G}_t) = 4\sigma^2\mathbb{E}_{X}(s(X)t(X))$
\end{theorem}

%\subsection{Asymptotic distribution of the empirical version of the test statistic}
%\label{SS:3.3}

In the case where the weight measure $\mu$ is equal to the law $P$ of the feature vector $X$, the test statistic takes the form
\begin{equation}\label{emp stat}
     \lambda^n_j = \int_{\mathcal{X}}\Big(\frac{\partial f_n(x)}{\partial x_j}\Big)^2dP(x)= \mathbb{E}_X\Big[\Big(\frac{\partial f_n(X)}{\partial x_j}\Big)^2\Big].
\end{equation}
In this case it is natural to estimate the expectation in \eqref{emp stat} using the empirical measure $P_n$ of the samples $\{X_i\}_{i=1}^n$,
\begin{equation}
\label{empirical test statistic}
    \hat{\lambda}_j^n = \int_{\mathcal{X}} \Big(\frac{\partial f_n(x)}{\partial x_j}\Big)^2dP_n(x) = n^{-1}\sum_{i=1}^n \Big(\frac{\partial f_n(X_i)}{\partial x_j}\Big)^2.
\end{equation}
The following result shows that the asymptotic distribution of the empirical test statistic $\hat{\lambda}_j^n$ is the same as that of $\lambda_j^n$ (when $\mu=P$).
%Hence, we would like to derive the asymptotic distribution of $r^2_n(\hat{\lambda}_j^n - \lambda_j)$ from $r^2_n(\lambda_j^n - \lambda_j)$.

\begin{proposition}[Asymptotic distribution of the empirical test statistic]
\label{corollary: empirical test statistic distribution}
Under the conditions of Theorem \ref{theorem: test statistic distribution}, we have
\begin{equation*}
    r_n^2 n^{-1} \sum_{i=1}^n \Big(\frac{\partial f_n(X_i)}{\partial x_j}\Big)^2 \Longrightarrow \mathbb{E}_X\Big[\Big(\frac{\partial h^{\star}(X)}{\partial x_j}\Big)^2\Big] = \int_{\mathcal{X}}\Big(\frac{\partial h^{\star}(x)}{\partial x_j}\Big)^2dP(x).
\end{equation*}
\end{proposition}

\section{Computing the Asymptotic Distribution}\label{S:4}
The limiting distribution of the test statistic in Theorem \ref{theorem: test statistic distribution} is a functional of the argmax of a Gaussian process over a function space. We discuss two approaches to computing this distribution in order to implement a test. One uses a series representation and the other a discretization method. Section \ref{S:5} numerically evaluates these approaches.

\subsection{Series Representation}
\label{SS:4.1}
%\subsubsection{Series representation}

%The limiting distribution of the test statistic is based on the argmax of a Gaussian process. To test the null using this distribution, we need to compute a quantile of this distribution which requires either to know the exact form of the distribution function or to at least be able to sample out of this distribution to estimate its quantile. But in this general setting, the distribution function of a a functional of the argmax of a Gaussian process over a function space is not known. 

We can exploit the fact that $\Theta$ is a subspace of the Hilbert space $L^2(P) = \{f: \mathcal{X} \to \mathbb{R}: \|f\|_{L^2(P)} < \infty\}$ which admits an orthonormal basis $\{\phi_i\}_{i=0}^{\infty}$. Assume that this basis is $C^1$ and is stable under differentiation. This means that $\forall i \in \mathbb{N}, 1 \leq j \leq d$, there exists an $\alpha_{i,j} \in \mathbb{R}$ and a mapping $k: \mathbb{N} \to \mathbb{N}$ such that
\begin{equation}
\label{basis function derivative}
    \frac{\partial \phi_i}{\partial x_j} = \alpha_{i,j} \phi_{k(i)}.
\end{equation}
This, in turn, implies that there exists an invertible operator $D$ which is diagonal with respect to the basis such that 
\begin{equation}
\label{stable basis}
  \|f\|^2_{k,2} = \|Df\|^2_{L^2(P)} = \sum_{i=0}^{\infty} d_i^2\langle f,\phi_i \rangle_{L^2(P)}^2,  
\end{equation}
where the $d_i$'s are certain functions of the $\alpha_{i,j}$'s that satisfy 
\begin{equation}
\label{condition on sum of weights}
    \sum_{i=0}^{\infty} \frac{1}{d_i^2} < \infty,\quad  \sum_{i=0}^{\infty} \frac{\alpha_{i,j}^2}{d_i^4} < \infty.
\end{equation}

\begin{theorem}[Series representation of the asymptotic distribution]
\label{theorem: new expression limiting distribution}
Assume the conditions of Theorem \ref{theorem: test statistic distribution} hold and that the weight measure $\mu=P$. If the orthonormal basis of $L^2(P)$ satisfies \eqref{basis function derivative}, then the asymptotic distribution of the test statistic satisfies
\begin{equation*}
     \mathbb{E}_X\bigg[\Big(\frac{\partial h^{\star}(X)}{\partial x_j}\Big)^2\bigg]\; = \; \frac{B^2}{\sum_{i=0}^{\infty} \frac{\chi_i^2}{d_i^2}}\sum_{i = 0}^{\infty} \frac{\alpha_{i,j}^2}{d_i^4} \chi_{i}^2, %\overset{d}{=}
\end{equation*}
where $\{\chi_i^2\}_{i \in \mathbb{N}}$ are i.i.d. samples from the chi-square distribution.
\end{theorem}

%\subsubsection{Example case: uniform input variables}
%\label{SS:4.2}

To provide a concrete example, we consider the case where the elements of the feature vector $X$ are mutually independent and distributed uniformly on $[-1,1]$. The density $p$ of $P$ takes the form
\begin{equation}
    \label{density input variables}
    p(x) = \frac{1}{2^d}.
\end{equation}
%Under the same conditions of Theorem \ref{theorem: test statistic distribution} and by assuming that the density of the input variables $p$ is defined as \eqref{density input variables}, 
We now use the Fourier basis, which is an orthonormal basis of $C^{\infty}$ functions for both $L_2([-1,1]^d)$ and $\mathbb{H}^{\lfloor \frac{d}{2} \rfloor+2}([-1,1]^d)$ 
with $$\|f\|^2_{\mathbb{H}^{\lfloor \frac{d}{2} \rfloor+2}} = 2^d\sum_{n \in \mathbb{N}^d} \Bigg[\sum_{|\alpha| \leq \lfloor \frac{d}{2} \rfloor+2} \prod_{k=1}^d\gamma_{n_k}^{\alpha_k} \Bigg]|f_n|^2,$$ where $\gamma_n = n^2 \pi^2$, $\phi_n^{[0]} = \cos(n\pi x)$ and $\phi_n^{[1]} = \sin(n\pi x)$.
We define the weights $d_n^{2}$ as $d_n^{2}=\sum_{|\alpha| \leq \lfloor \frac{d}{2} \rfloor+2} \prod_{k=1}^d\gamma_{n_k}^{\alpha_k}$. This implies that
\begin{equation*}
    h^{\star}_{\omega} = \sum_{i \in \{0,1\}^d}\sum_{n \in \mathbb{N}^d} \frac{B}{\|D^{-1}\phi_{\xi}\|}\frac{\xi_n^{[i]}}{d_n^{[i]2}} \phi_n^{[i]},
\end{equation*}
where $\|D^{-1}\phi_{\xi}\|^2 = 4\sigma^2\sum_{i \in \{0,1\}^d}\sum_{n \in \mathbb{N}^d} \frac{\chi_{n}^{[i]2}}{d_n^{2}}$ and $\{\chi_{n}^{[i]2}\}_{n \in \mathbb{N}^d, i \in \{0,1\}^d}$ are i.i.d. samples from a chi-square distribution.  It follows that

\begin{equation}
\label{estimable distribution}
      \mathbb{E}_X\Big[\Big(\frac{\partial h^{\star}(X)}{\partial x_j}\Big)^2\Big] = \Big\|\frac{\partial h^{\star}(X)}{\partial x_j}\Big\|^2_{L^2(P)} = \frac{B^2}{\sum_{i \in \{0,1\}^d}\sum_{n \in \mathbb{N}^d} \frac{\chi_{n}^{[i]2}}{d_n^{2}}}\sum_{i \in \{0,1\}^d}\sum_{n \in \mathbb{N}^d} \frac{\gamma_{n_j}}{d_n^{4}} \chi_{n}^{[i]2}.
\end{equation}
%with $\gamma_n = n^2 \pi^2$ and  $d_n^{2} = \sum_{|\alpha| \leq \lfloor \frac{d}{2} \rfloor+2} \prod_{k=1}^d\gamma_{n_k}^{\alpha_k}$ and $\{\chi_{n}^2\}_{n \in \mathbb{N}^d}$ are i.i.d. chi-squared variables.

%\subsubsection{Estimation of the series representation and implementation of the test}
%\label{SS:4.3}
%To perform the test in practice, it is necessary to compute the quantile at a chosen significance level $\alpha$ of the continuous asymptotic distribution $F$ derived in Theorem \ref{theorem: new expression limiting distribution}. Since the distribution $F_j$ is expressed as a mixture of chi-squared variables, it can be approximated thanks to a finite order truncation of the sums. The finite order can be chosen to achieve a chosen degree of accuracy. 

The implementation of the test requires computing a quantile of the limiting distribution. We can approximate the limiting distribution by truncating the infinite sum in Theorem \ref{theorem: new expression limiting distribution} at some order $N$, which can be chosen to achieve some given degree of accuracy. By sampling from a chi-squared distribution, we generate $m$ samples $\{Z_i^N\}_{i=1}^m$ from an approximate asymptotic distribution, $F^N$. If we let $m$ grow along with $N$, i.e., if we consider a sequence $m_N$ such that $m_N \to \infty$ as $N \to \infty$, we obtain the consistency of the empirical approximate distribution, denoted $\mathbb{F}_{m_N}^{N}$,
\begin{equation}
\label{consistency empirical distribution}
    \mathbb{F}_{m_N}^{N}(x) \xrightarrow[N \to \infty]{P} F(x), \quad \forall x \in \mathbb{R}.
\end{equation}
This follows from a triangular array weak law of large number.
Define the empirical approximate quantile function $\mathbb{F}_{m_N,N}^{-1}$ as
\begin{equation}
\label{empirical quantile function}
    \mathbb{F}_{m_N,N}^{-1}(\alpha) = Z^{N}_{m_N(i)},\quad \mbox{ } \text{for} \mbox{ } \alpha \in \Bigg(\frac{i-1}{m_N},\frac{i}{m_N}\Bigg],
\end{equation}
where $Z^{N}_{m_N(1)},...,Z^{N}_{m_N(n)}$ are the order statistics of $\{Z_i^N\}_{i=1}^{m_N}$.
Then, from the previous result and Lemma 21.2 of \cite{van2000asymptotic}, we get the consistency of the empirical approximate quantile function:
\begin{equation}
\label{consistency empirical quantile}
    \mathbb{F}_{m_N,N}^{-1}(\alpha) \xrightarrow[N \to \infty]{P} F^{-1}(\alpha),\quad \forall 0 \leq \alpha \leq 1.
\end{equation}
Thanks to this consistent estimator, by letting the truncation order $N$ go to infinity as $n \to \infty$, we can reject the null if the empirical test statistic $\hat{\lambda}_n > \mathbb{F}_{m_N,N}^{-1}(1-\alpha)$ such that the test will be asymptotically of level $\alpha\in[0,1]$
\begin{equation*}
    \mathbb{P}_{H_0}\big(\hat{\lambda}_n > \mathbb{F}_{m_N,N}^{-1}(1-\alpha)\big) \leq \alpha.
\end{equation*}

The truncation order $N$ can be chosen so as to achieve a given approximation error $\epsilon$. We illustrate this for the choice (\ref{density input variables}) discussed above. From Fourier approximation theory, we know that given the full index set of level $N$, $I_N = \{n \in \mathbb{N}^d: \max_{1\leq j \leq d} n_j \leq N-1\}$, the approximation rate of the Fourier series of order $N$ for any $f \in \Theta$ is
\begin{equation}
\label{approximation rate Fourier series}
    \Big\|f - \sum_{i \in \{0,1\}^d}\sum_{n \in I_N} \langle f, \phi_n^{[i]} \rangle\phi_n^{[i]}\Big\|_{L^2} \leq O\Bigg(\frac{1}{N^{\lfloor \frac{d}{2} \rfloor+2}}\Bigg).
\end{equation}
This entails that $O\Big(\frac{1}{\epsilon}\Big)^{\frac{d}{\lfloor \frac{d}{2} \rfloor+2}}$ Fourier basis functions will be necessary to reach a given $L^2$ approximation error $\epsilon$.

\subsection{Discretization Approach}%{Discretization based estimation of the asymptotic distribution}
\label{SS:4.4}

The limiting distribution of the test statistic is a functional of the argmax $h^{\star}$ of a Gaussian process indexed by the function space $\Theta$. The argmax is defined as a random function such that $\forall \omega \in \Omega$, $h^{\star}_{\omega}$ satisfies $\mathbb{G}_{h^{\star}_{\omega}}(\omega) \geq \mathbb{G}_f(\omega)$, $\forall f \in \Theta$. Given an $\epsilon$-cover $\{f_1,...,f_C\}$ of the function space $\Theta$ of size $C$, a discrete approximation of the paths of the process can be obtained by sampling from a multivariate Gaussian variable of dimension $C$ with mean 0 and covariance matrix whose elements are $\mathbb{E}_X(f_i(X)f_j(X))_{i,j=1}^C$. For every $\omega$ corresponding to a sample from this multivariate Gaussian, the argmax among the $\epsilon$-cover can be identified and by applying the functional (\ref{functional def}) to it, this generates an approximate sample from the limiting distribution. The accuracy is controlled by $\epsilon$.

Computing an explicit $\epsilon$-cover of the function space $\Theta$ can be challenging. A simple approach to approximate this cover uses random functions sampled from it. More specifically, we generate random neural networks (\ref{NN}) by sampling the network parameter. By generating a large enough number of networks, we increase the likelihood of covering $\Theta$.

\section{Simulation Experiments}
\label{S:5}
This section provides numerical results that illustrate our theoretical results and the properties of the significance test.

\subsection{Data-Generating Process}
We consider a vector $X$ of eight feature variables satisfying $$X=(X_1,\ldots,X_8) \sim \mathcal{U}(-1, 1)^8.$$
We consider the following data generating process,
\begin{align}\label{model}
Y = 8 + X_1^2 + X_2X_3 + \cos(X_4) + \exp{(X_5X_6)} + 0.1X_7 + \epsilon,
\end{align}
where $\epsilon \sim N(0, \sigma^2)$ and $\sigma = 0.01$. The variable $X_8$ has no influence on $Y$ and is hence irrelevant.
We generate a training set of $n=100,000$ independent samples and validation and testing sets of $10,000$ independent samples each.

\subsection{Fitting a Neural Network}
We fit a fully-connected feed-forward neural network with one hidden layer and sigmoid activation function to the training set using the TensorFlow package. We employ the Adam stochastic optimization method with step size 0.001 and exponential decay rates for the moment estimates of 0.9 and 0.999. We use a batch size of 32, a maximum number of 150 epochs and an early stopping criterion that stops the training when the validation error has not decreased by more than $10^{-5}$ for at least 5 epochs. The number of hidden nodes is chosen so as to minimize the validation loss. A network with 25 hidden units performs best. 

Table \ref{table: MSEs} compares the neural network's mean square error (MSE) for the test set with the test MSE for a linear model fitted to the same training data set. The neural network's test MSE is of the order of the regression noise. The linear model's test MSE is three orders of magnitude larger.

\begin{table}[t!]
\centering
 \begin{tabular}{| c | c |}
    \hline
    \textbf{Model} & \textbf{Test MSE} \\ \hline
    Neural Network & $3.1 \cdot 10^{-4}$ \\\hline
    Linear & 0.35  \\ 
    \hline
\end{tabular}
\caption{Test MSE for the neural network and an alternative linear model.}
\label{table: MSEs}
\end{table}

\subsection{Test Statistic}

For each variable $1\leq j \leq 8$, we compute the empirical test statistic $\hat{\lambda}^n_j$ in (\ref{empirical test statistic}) using the \textsf{gradients} function of TensorFlow over the training set. %: $$\hat{\lambda}_j^n= \frac{1}{n}\sum_{i=1}^n \Big(\frac{\partial f_n(X_i)}{\partial x_j}\Big)^2.$$ 
Table \ref{table: test statistic} reports the results.
\begin{table}[t]
\centering
 \begin{tabular}{| c | c | c |}
    \hline
    \textbf{Variable} & \textbf{Test Statistic}  & \textbf{Leave-One-Out Metric}\\ \hline
    $X_1$ & 1.31 & 8.94$\cdot 10^{-2}$ \\ \hline
    $X_2$ & 0.332 & 1.12$\cdot 10^{-1}$ \\ \hline
    $X_3$ & 0.331 & 1.12$\cdot 10^{-1}$\\ \hline
    $X_4$ & 0.267 & 2.09$\cdot 10^{-2}$\\ \hline
    $X_5$ & 0.480 & 1.30$\cdot 10^{-1}$\\ \hline
    $X_6$ & 0.479 & 1.30$\cdot 10^{-1}$\\ \hline
    $X_7$ & 1.01$\cdot 10^{-2}$ & 3.50$\cdot 10^{-3}$ \\ \hline
    $X_8$ & 4.20$\cdot 10^{-6}$ & 3.46$\cdot 10^{-6}$ \\
    \hline
\end{tabular}
\caption{Values of empirical test statistic $\hat{\lambda}^n_j$ and leave-one-out metric.}
\label{table: test statistic}
\end{table}
The statistic properly discriminates the irrelevant variable $X_8$ from the relevant variables. We also note that the statistic is similar for the variables that have a symmetric influence, namely $(X_2, X_3)$ and $(X_5, X_6)$. Moreover, for the linear variable $X_7$, the statistic $\hat{\lambda}_7^n=0.01=0.1^2$, showing that the statistic correctly identifies the coefficient of $X_7$, which is $0.1$. We note that the test statistic can be used to rank order variables according to their influence. The greater the statistic the greater the influence. Thus the quadratic variable $X_1$ is identified as the most important variable.  

For comparison, Table \ref{table: test statistic} also reports the leave-one-out metric, which is widely used in practice to assess variable influence. The leave-one-out metric for variable $X_j$ is the difference between the loss of the model based on all variables except $X_j$ and the loss of the model based on all variables. It is generally positive; the larger its value the larger one deems the influence of a variable. The results in Table \ref{table: test statistic} suggest that the leave-one-out metric shares the symmetry properties with our test statistic. There are important differences, however. For example, relative to our test statistic, the leave-one-out metric discounts the influence of the quadratic variable $X_1$. It identifies the exponential variables $X_5$ and $X_6$ as the most important variables. Moreover, the calculation of the leave-one-out metric requires re-fitting the model $d$ times, where $d$ is the dimension of $X$. This can be computationally expensive for models with many features. The calculation of our test statistic does not require re-fitting the model. Perhaps most importantly, our test statistic is supported by a rigorous test. To our knowledge, no test has been provided for the leave-one-out metric. 
%there is little statistical reasoning behind the leave-one-out metric.

\subsection{Estimation of Quantile}
We choose a confidence level $\alpha = 0.05$ and compute the empirical quantile $\mathbb{F}_{m_N,N}^{-1}(1-\alpha)$ using the series representation approach described in Section \ref{SS:4.1}. We use the Fourier basis, choose a truncation order $N = 4$, and sample $m_N=10,000$ observations from the corresponding approximate asymptotic distribution.

We also need to choose the constant $B$ that uniformly bounds the Sobolev norm of the functions in $\Theta$ as defined in \eqref{def function space}. The asymptotic distribution in \eqref{estimable distribution} is scaled by $B$. We need to choose $B$ large enough so that the unknown regression function $f_0$ belongs to $\Theta$, which translates to $\|f_0\|_{\lfloor\frac{d}{2}\rfloor+2} \leq B$. But because a $B$ that is too large would result in a test with smaller power, the optimal value of $B$ is the norm of $f_0$. One could estimate this norm using Monte Carlo simulation or numerical integration, but because this may be computationally expensive, we follow the following alternative approach.

We include several uniformly distributed auxiliary noise variables in $X$ and fit the neural network using both the original variables as well as the noise variables. The test statistics for the noise variables have a limiting distribution that includes the $B$ factor. We also estimate the quantile of the asymptotic distribution as described in Section \ref{SS:4.1} by sampling from the limiting distribution with value $B = 1$, hence these are unscaled samples. An estimator of $B$ can be obtained by computing the ratio of the mean of the test statistics of the noise variables to the mean of the unscaled samples. Because the number of noise variables is relatively small, this estimator can be noisy. Therefore, we prefer to use the maximum value of the test statistics of the noise variables (rather than the average) to increase the likelihood that $\|f_0\|_{\lfloor\frac{d}{2}\rfloor+2} \leq B$.

\subsection{Results}
We estimate the power and size of the test by performing it on 250 alternative data sets $(Y_i,X_i)_{i=1}^n$ generated from the model (\ref{model}). The second column of Table \ref{table: power and size} reports the results. The power of the test is ideal (i.e., there is no Type 2 error). The size (Type 1 error) is 0.35, larger than the $\alpha=0.05$ level of the test. The approximate asymptotic distribution on which the test is based might underestimate the variance of the finite sample statistic. 

\begin{table}[t!]
\centering
 \begin{tabular}{| c | c | c | c |}
    \hline
    \textbf{Variable} & \textbf{NN Test (S)} & \textbf{NN Test (D)} & \textbf{$t$-Test} \\ \hline
    $X_1$  & 1 & 1 & 0.07\\ \hline
    $X_2$  & 1 & 1 & 0.07\\ \hline
    $X_3$  & 1 & 1 & 0.06\\ \hline
    $X_4$  & 1 & 1 & 0.05\\ \hline
    $X_5$  & 1 & 1 & 0.10\\ \hline
    $X_6$  & 1 & 1 & 0.07\\ \hline
    $X_7$  & 1 & 1 & 1\\ \hline
    $X_8$  & 0.35 & 0 & 0.04\\
    \hline
\end{tabular}
\caption{Power and size of significance tests at 5\% level. The NN Test (S) column reports the values for our test when the series method described in Section \ref{SS:4.1} is used to compute the asymptotic distribution of the test statistic. The NN Test (D) column reports the values for our test when the discretization method described in Section \ref{SS:4.4} is used to compute the asymptotic distribution. The $t$-Test column reports the values for a standard $t$-test in an alternative linear regression model. }
\label{table: power and size}
\end{table}

The last column of Table \ref{table: power and size} also reports power and size of a $t$-test for an alternative linear regression model. Only the linear variable $X_7$ is identified as significant. These results illustrate how model misspecification can hurt inference.

%with significance based on t-test obtained from the linear regression in Table \ref{table: OLS}. Interestingly, only the intercept and the seventh variable (that is the only linear one) are considered as significant from the liner model. This shows how model misspecification can hurt inference. 
%\begin{table}[ht]
%\centering
%\caption{Summary of Linear Regression}
%\label{table: OLS}
% \begin{tabular}{| c | c | c | c | c | c | c |}
%    \hline
%    \textbf{Variable} & \textbf{coef}& \textbf{std err}& \textbf{t}& \textbf{$P > |t|$}& \textbf{[0.025}& \textbf{0.975]} \\ \hline
%    const &	10.2297	&0.002&	5459.250&	0.000&	10.226&	10.233\\\hline
%1&	-0.0031&	0.003&	-0.964&	0.335&	-0.009&	0.003\\\hline
%2&	0.0051&	0.003	&1.561&	0.118&	-0.001&	0.011\\\hline
%3&	-0.0026&	0.003&	-0.800&	0.424&	-0.009&	0.004\\\hline
%4&	0.0003&	0.003	&0.085&	0.932&	-0.006&	0.007\\\hline
%5&	0.0016&	0.003	&0.493&	0.622&	-0.005&	0.008\\\hline
%6&	-0.0033&	0.003&	-1.035&	0.300&	-0.010&	0.003\\\hline
%7&	0.0976	&0.003&	30.059&	0.000&	0.091&	0.104\\\hline
%8&	-0.0018	&0.003&	-0.563&	0.573&	-0.008&	0.005\\\hline
%9&	-0.0024	&0.003&	-0.733&	0.463&	-0.009&	0.004\\\hline
%10&	0.0063	&0.003&	1.937&	0.053&	-7.36e-05&	0.013\\
%\hline
%\end{tabular}
%\end{table}

\subsection{Correlated Feature Variables}
In practice, the elements of the feature vector $X$ may be correlated. We study the robustness of the test in this regard. 
%For instance, the Leave-One-Out metric is well known to fail capturing the predictive power of highly correlated variables. The goal of this subsection is to assess the robustness of our test statistic against variables correlation. 
We generate correlated variables $X_i$ whose marginal distribution is still $\mathcal{U}(-1,1)$ but whose correlation structure is controlled by a Gaussian copula. We first generate samples from a multivariate normal distribution $(Z_1,\ldots,Z_8) \sim \mathcal{N}(0, \Sigma)$ with covariance matrix $\Sigma$ given by

\[\Sigma =
\begin{bmatrix}
    1 & 0.1 & 0 & 0 & 0 & 0 & 0 & 0 \\
    0.1 & 1 & 0 & 0 & 0 & 0 & 0 & 0 \\
    0 & 0 & 1 & 0 & 0 & 0 & 0 & 0 \\
    0 & 0 & 0 & 1 & 0 & 0 & 0.3 & 0 \\
    0 & 0 & 0 & 0 & 1 & 0.5 & 0 & 0 \\
    0 & 0 & 0 & 0 & 0.5 & 1 & 0 & 0 \\
    0 & 0 & 0 & 0.3 & 0 & 0 & 1 & 0 \\
    0 & 0 & 0 & 0 & 0 & 0 & 0 & 1 \\
\end{bmatrix},
\]
implying that $X_1$ and $X_2$, $X_5$ and $X_6$ as well as $X_4$ and $X_7$ are correlated. We then apply the standard normal distribution function on each coordinate to obtain the corresponding correlated uniform variables. We generate new data sets and re-fit the model.

%\begin{table}[t!]
%\centering
% \begin{tabular}{| c | c | c |}
%    \hline
%    \textbf{Variable} & \textbf{Test Statistic} & \textbf{Leave-One-Out Metric} \\ \hline
%    $X_1$ & 2$\cdot 10^{-1}$\% & 797\% \\ \hline
%    $X_2$ & 9$\cdot 10^{-2}$\% & 572\% \\ \hline
%    $X_3$ & 4$\cdot 10^{-1}$\% & 565\% \\ \hline
%    $X_4$ & 2$\cdot 10^{-1}$\% & 2098\% \\ \hline
%    $X_5$ & 58\% & 147\% \\ \hline
%    $X_6$ & 57\% & 148\% \\ \hline
%    $X_7$ & 8$\cdot 10^{-1}$\% & 95\% \\ \hline
%    $X_8$ & 16\% & 873\% \\ 
%    \hline
%\end{tabular}
%\caption{Relative percentage change in the test statistic and leave-one-out metric (vs. values in Table \ref{table: test statistic}) due to correlation of selected variables.}
%\label{table: difference importance values}\end{table}

\begin{table}[t!]
\centering
 \begin{tabular}{| c | c |}
    \hline
    \textbf{Variable} & \textbf{NN Test (D)} \\ \hline
    $X_1$  & 1 \\ \hline
    $X_2$  & 1 \\ \hline
    $X_3$  & 1 \\ \hline
    $X_4$  & 1 \\ \hline
    $X_5$  & 1 \\ \hline
    $X_6$  & 1 \\ \hline
    $X_7$  & 1 \\ \hline
    $X_8$  & 0.004 \\
    \hline
\end{tabular}

\caption{Power and size of significance tests at 5\% level when the feature variables are correlated. The discretization method described in Section \ref{SS:4.4} was used to compute the asymptotic distribution of the test statistic.}
\label{table: power and size correlation}\end{table}

Table \ref{table: power and size correlation} reports the size and power of the test in this situation (as above, 250 alternative data sets were used). The size and power of the test do not deteriorate, suggesting that the test performs well in the presence of correlated feature variables.

%Our test statistic appears to be much more robust with respect to the correlation than the leave-one-out metric, suggesting that the test performs relatively well in the presence of correlated feature variables.

\subsection{Discretization Approach}\label{discrete}
We study the performance of the significance test when using the discretization approach in Section \ref{SS:4.4} to estimate the asymptotic distribution. This approach approximates the argmax of the Gaussian process by using random functions sampled from $\Theta$. More specifically, we randomly sample 500 functions using fully-connected, single-layer feed-forward neural networks with sigmoid activation function and 25 hidden units. The parameters are sampled from a Glorot normal distribution, which is a truncated normal distribution centered at 0 with standard deviation $\sqrt{2/(in + out)}$, where $in$ is the number of input nodes and $out$ is the number of output nodes of the layer for which the parameters are sampled (see \cite{glorot2010understanding}).

To generate a sample from the asymptotic distribution, we first generate a sample from a multivariate normal distribution of dimension 500 with mean vector 0 and covariance matrix that we approximate by the diagonal matrix with diagonal elements equal to $\sigma_j^2 = \frac{1}{n}\sum_{i=1}^n f_j(X_i)^2$. The value $\sigma_j^2$ is the empirical expectation of the square of the function that we estimate. We extract the index of the maximum value from this multivariate normal, which correspond to the approximate argmax of the Gaussian process. Given this approximate argmax function, we finally generate an approximate sample from the asymptotic distribution of the test statistic. We repeat this process 10,000 times in order to estimate the quantile at the 5\% level.

The third column of Table \ref{table: power and size} reports the power and size of the test based on this approach (as before, we use 250 data sets to estimate the power and size). Our significance test has perfect power and size.

\section{Empirical Application: House Price Valuation}
\label{S:6}
We use the significance test to study the variables influencing house prices in the United States.
%\subsection{Data}
We analyze a data set of 76,247 housing transactions in California's Merced County between 1970 and 2017. The data are obtained from the county's registrar of deed office through the data vendor CoreLogic. The dependent variable $Y$ of the regression is the logarithm of the sale price. The 68 elements of the feature vector $X$ are listed in Appendix \ref{table: variables descriptions} along with descriptions. They include house characteristics, mortgage characteristics, tax characteristics, local and national economic conditions, and other types of variables. Every variable is centered and scaled to unit variance. $70\%$ of the data is used for training, $20\%$ for validation, and the remainder is used for testing. 

We fit a fully-connected, single-layer feed-forward neural network with sigmoid activation function and $l$-1 regularization term using TensorFlow. We employ the Adam stochastic optimization method with step size of 0.001 and exponential decay rates for the moment estimates of 0.9 and 0.999. We use a batch size of 32, a maximum number of 100 epochs and an early stopping criterion that stops the training when the validation error has not decreased by more than $10^{-3}$ for at least 10 epochs. The number of hidden nodes and the regularization weight are chosen so as to minimize the validation loss. The optimal network architecture has 150 hidden units and the optimal regularization weight is $10^{-5}$. The mean squared error (MSE) for the test set is $0.45$.

%Table \ref{table: MSEs house data} reports the mean squared error (MSE) for the test set. over the test set of the fitted network is compared with the test MSE of a linear model fitted on the same training set 
%\begin{table}[ht]
%\centering
%\caption{Comparison of the out-of-sample MSE of the neural network with a linear model for house prices data}
%\label{table: MSEs house data}
% \begin{tabular}{| r | c |}
%    \hline
%    \textbf{Model} & \textbf{Mean Squared Error} \\ \hline
%    Linear & 0.85  \\ \hline
%    Neural Network & 0.60 \\
%   \hline
%\end{tabular}
%\end{table}

%\subsection{Results}

\begin{table}[t!]
\centering
 \begin{tabular}{| c | c |}
    \hline
    \textbf{Variable Name} & \textbf{Test Statistic} \\ \hline
Last\_Sale\_Amount	& 1.640 \\ \hline
Tax\_Land\_Square\_Footage & 1.615 \\ \hline
Sale\_Month\_No &	1.340 \\ \hline
Tax\_Amount	& 0.3834 \\ \hline
Last\_Mortgage\_Amount & 0.1040 \\ \hline
Tax\_Assd\_Total\_Value &	0.0812 \\ \hline
Tax\_Improvement\_Value\_Calc & 0.0721 \\ \hline
Tax\_Land\_Value\_Calc &	0.0690 \\ \hline
Year\_Built	& 0.0681 \\ \hline
SqFt & 0.0566 \\ \hline
\end{tabular}
\caption{Test statistic for the top 10 significant ($5\%$) feature variables (out of 68).}
\label{table: test statistic house data top 10}
\end{table}

We use the discretization method to compute the asymptotic distribution of the test statistic (see Section \ref{SS:4.4} above). We sample 500 random functions using neural networks with the same architecture as the fitted one but with weights sampled from a Glorot normal distribution (see Section \ref{discrete} above). 10,000 samples are generated to estimate the quantile of the asymptotic distribution.

Table \ref{table: test statistic house data top 10} reports the values of the empirical test statistic (\ref{emp stat}) for the top 10 variables that are significant at the $5\%$ level; the variables are ordered according to the magnitude of the test statistic. The last sale price is the most influential variable, closely followed by the square footage of the land. The month of sale is also very influential. This reflects the seasonality of sales: the spring months usually bring the most listings and the best sales prices, as many people try to move in the summer during summer holidays. The square footage of the home turns out to be less influential than perhaps expected.

Table \ref{table: test statistic house data bottom 10} reports the variables that are found to be insignificant at the $5\%$ level. 

\begin{table}[t]
\centering
 \begin{tabular}{| c | c |}
    \hline
    \textbf{Variable Name} & \textbf{Test Statistic} \\ \hline
%Median\_SAF\_PriceChange\_l12 & 0.0019 \\ \hline
%Foreclosures\_To\_Loans\_Ratio	& 0.0017 \\ \hline
%Median\_SAF\_PriceChange\_Model\_l12 & 0.0016 \\ \hline
%Median\_NonInst\_PriceChange\_Model\_l12 & 0.00095 \\ \hline
%Median\_SAF\_PriceChange\_Model\_pct &	0.00017 \\ \hline
%Median\_SAF\_PriceChange\_pct	& 0.00014 \\ \hline
Median\_NonInst\_PriceChange\_pct & 4.430e-05 \\ \hline
Median\_SAF\_PriceChange	& 3.096e-05 \\ \hline
Median\_SAF\_PriceChange\_Model & 2.795e-05 \\ \hline
Median\_NonInst\_SAF\_PriceChange\_Model\_pct & 1.405e-05 \\ \hline
\end{tabular}
\caption{Variables that are found to be insignificant at the $5\%$ level.}
\label{table: test statistic house data bottom 10}
\end{table}

%reports the bottom 10 variables. At the $5\%$ level, the corresponding averaged quantile  over all variables is equal to $0.00042$. This implies that the last four variables: Median\_NonInst\_PriceChange\_pct, Median\_SAF\_PriceChange, Median\_SAF\_PriceChange\_Model and \\ Median\_NonInst\_SAF\_PriceChange\_Model\_pct are considered as insignificant. All other variables are found to have a significant influence on house prices. 

%The month of sale appears to be the most influential variable. This finding reflects the seasonality of sales: the spring months usually bring the most listings and the best sales prices, as many people try to move in the summer when the school is on break.  The last sale amount is also influential, along with the number of past sales of the property. The square footage, the number of bedrooms and the number of bathrooms turn out to be less influential than perhaps expected.

\section{Conclusion}\label{S:7}
%While neural networks are successfully used in many fields, they are often considered as black-box models that permit little insight into how the individual feature variables influence a prediction. 
We develop and analyze a pivotal test for assessing the statistical significance of the feature variables of a single-layer, fully connected feedforward neural network that models an unknown regression function. The gradient-based test statistic is a weighted average of the squared partial derivative of the neural network estimator with respect to a variable. We show that under technical conditions, the large-sample asymptotic distribution of the test statistic is given by a mixture of chi-square distributions. A simulation study illustrates the performance of the test and an empirical application to house price valuation highlights its behavior given real data.  

% Acknowledgements should go at the end, before appendices and references

\acks{We are grateful for comments from the participants of the 2018 NeurIPS conference, the Conference on Interpretable Machine-Learning Models and Financial Applications at MIT, the Machine Learning in Finance Conference at Oxford University, the Machine Learning in Finance Workshop at Columbia University, the Machine Learning Conference at EPFL Lausanne, the Risk Seminar at ETH Zurich, and the OR Seminar at UC Louvain. We are also grateful for discussions with Jose Blanchet, Tze Lai, and Yinyu Ye. We finally would like to thank Yerso Checya and Stefan Weber for pointing out gaps in the proofs of Theorem \ref{theorem: NN distribution} and Proposition \ref{corollary: empirical test statistic distribution} and helping to fix them.}

\appendix
\section{Proofs}

The following result is used repeatedly in the proofs.
\begin{theorem}[Theorem 2.11.23 of \cite{van1996weak}]
For each $n$, let $\mathcal{F}_n = \{f_{n,t}: t \in T\}$ be a class of measurable functions indexed by a totally bounded semimetric space $(T, \rho)$. Given envelope functions $F_n$, assume that
$$\mathbb{E}(F_n^2) = O(1),$$
$$\mathbb{E}(F_n^21_{\{F_n > \eta\sqrt{n}\}}) \to 0, \textsl{for every } \eta >0,$$
$$\sup_{\rho(s,t) < \delta_n} \mathbb{E}[(f_{s,t} - f_{n,t})^2] \to 0, \textsl{for every } \delta_n \to 0.$$
If 
$$\int_0^{\delta_n} \sqrt{\log N_{[ ]}(\epsilon\|F_n\|_{P,2},\mathcal{F}_n, L^2(P))}d\epsilon \to 0, \textsl{for every } \delta_n \to 0,$$
then the sequence $\{\mathbb{G}_n f_{n,t}: t \in T\}$ is asymptotically tight in $l^{\infty}(T)$ and converges in distribution to a tight Gaussian process provided the sequence of covariance functions $\mathbb{E}(f_{n,s}f_{n,t}) - \mathbb{E}(f_{n,s})\mathbb{E}(f_{n,t})$ converges pointwise on $T \times T$.
\end{theorem}

\begin{proof}[Proof of Theorem \ref{theorem: NN distribution}]

The proof proceeds in three steps. First we derive the estimation rate $r_n$ of $f_n$. Then we show that a rescaled and shifted version of the empirical criterion function converges in distribution to a Gaussian process. Finally we apply the argmax continuous mapping theorem.

First we consider the estimation rate $r_n$ of $f_n$ which implies the tightness of the random sequence $h_n = r_n(f_n - f_0)$. \cite{chen1998sieve} derive an estimation rate with respect to the $L^2(P)$-metric on $\Theta$, $\|f - g\|^2 = \mathbb{E}_X[(f(X)-g(X))^2]$ of the neural network sieved M-estimator, i.e.  $\|f_n - f_0\| = O_P\big(\frac{1}{r_n}\big)$ with $K_n$ such that $K_n^{2+1/d}\log K_n = O(n)$. 

The second step is to show that a rescaled and shifted version of the empirical criterion function converges in distribution to a Gaussian process. We use a CLT for empirical processes on classes of functions changing with $n$, see Theorem 2.11.23 of \cite{van1996weak}.

Using the definition of the regression problem in \eqref{regression setting}, let us redefine the criterion function as $l_{g}(X, \epsilon) = 2(g-f_0)(X)\epsilon - (f_0 - g)^2(X)$. We now prove a CLT for the empirical processes defined on the class of function $\{r_n(l_{f_0 + h/r_n} - l_{f_0}): h \in K\}$ for every compact subset $K \subset \Theta$ by using \cite[Theorem 2.11.23]{van1996weak}. $K$ is a compact metric space with associated metric $\rho(s,t) = \|s-t\|_{L^2(P)}$ where $P$ denotes the joint law of $X$ and $\epsilon$.

We define 
\begin{equation*}
    f_{n,h} = r_n(l_{f_0 + h/r_n} - l_{f_0}) = 2h(X)\epsilon -\frac{h(X)^2}{r_n},
\end{equation*}
and the associated class of functions
\begin{equation*}
    \mathcal{F}_n = \{f_{n,h}: h \in K\}.
\end{equation*}

The envelope functions $F_n$ of $\mathcal{F}_n$ must be such that $\forall h \in K$, $f_{n,h}(x,\epsilon) \leq F_n(x,\epsilon)$, hence $F_n$ can be chosen as $(x, \epsilon) \to \sup_{h \in K} f_{n,h}(x, \epsilon)$. 
Using \cite[Lemma 2]{chen1998sieve}, we have that $||h||_{\infty} \lesssim ||h||_2^c$ where $c= \frac{2}{2+d}$. Since $K$ is a compact metric space in the $L^2$-norm, it is bounded which means that $\exists 1 < M < \infty$ such that $\forall h \in K$, $\|h\|_2 < M$.

Hence 
\begin{equation*}
    |f_{n,h}| \leq 2|\epsilon||h(x)| + \frac{1}{r_n}|h(x)|^2 \leq 2|\epsilon|M + \frac{1}{r_n}M^2,
\end{equation*}
and therefore, we can chose the envelope function as
\begin{equation*}
    F_n(x,\epsilon) = 2|\epsilon|M + \frac{2}{r_n}M^2,
\end{equation*}
where the additional factor 2 is added to simplify the computation of the metric entropy.

By definition of the envelope function and because we assume finite second moment for the regression error $\epsilon$, we have 
$$\mathbb{E}[F_n^2(X, \epsilon)] = O(1),$$ 
and 
$$\mathbb{E}(F_n^2 1_{\{F_n > \eta \sqrt{n}\}}) \to 0, \forall \eta > 0.$$

In addition,
\begin{align*}
(f_{n,h_1} - f_{n,h_2})^2 &= \Big[\Big(2\epsilon - \frac{1}{r_n}(h_1 + h_2)\Big)(h_1-h_2)\Big]^2 \leq \Big[\Big(2|\epsilon| + \frac{2}{r_n}M\Big)(h_1-h_2)\Big]^2.   
\end{align*}
By taking expectations, we obtain
$$\mathbb{E}[(f_{n,h_1} - f_{n,h_2})^2] \lesssim \mathbb{E}[(h_1-h_2)^2],$$ which implies that

$$\sup_{\rho(h_1,h_2)< \delta_n} \mathbb{E}(f_{n,h_1} - f_{n,h_2})^2 \leq O(\delta_n^2),$$

$$\sup_{\rho(h_1,h_2)< \delta_n} \mathbb{E}(f_{n,h_1} - f_{n,h_2})^2 \to 0, \forall \delta_n \downarrow 0.$$

Now, we need to find an upper bound on the metric entropy with bracketing of $\mathcal{F}_n$, 

$N_{[]}(\epsilon||F_n||_{P,2}, \mathcal{F}_n, L^2(P))$, to be able to prove that $$\int_0^{\delta_n} \sqrt{\log N_{[]}(\epsilon||F_n||_{P,2}, \mathcal{F}_n, L^2(P))}d\epsilon \to 0, \forall \delta_n \downarrow 0.$$

To do so, we will upper bound the metric entropy with bracketing associated with $\mathcal{F}_n$ with the metric entropy without bracketing of the initial function space $\Theta$.
Indeed, 
\begin{align*}
|f_{n,h_1} - f_{n,h_2}| &= |h_1-h_2|\Big|2\epsilon - \frac{1}{r_n}(h_1 + h_2)\Big|\\
& \leq |h_1-h_2|\Big(2|\epsilon| + \frac{2}{r_n}M\Big) = |h_1-h_2|F_n(x, \epsilon)\\
&\leq ||h_1-h_2||_{\infty}F_n(x, \epsilon).
\end{align*}
The last inequality implies that
\begin{equation}
\label{entropy inequality}
    N_{[]}(\epsilon||F_n||_{P,2}, \mathcal{F}_n, L^2(P)) \leq N\Big(\frac{\epsilon}{2}, \Theta, \|\cdot\|_{\infty}\Big).
\end{equation}

Thanks to \cite[Corollary 2, Corollary 4]{nickl2007bracketing} and because $\Theta$ as defined in \eqref{def function space} is a subset of the weighted Sobolev space $\textsc{H}^{s}_2(\mathbb{R}^d, \langle x \rangle^3)|\mathcal{X}$ restricted to the input space with $s = \lfloor \frac{d}{2}\rfloor+2$, we have that
\begin{equation}
\label{upper bound metric entropy}
    \log N\Big(\frac{\epsilon}{2}, \Theta, \|\cdot\|_{\infty}\Big) \lesssim \Big(\frac{1}{\epsilon}\Big)^{\frac{d}{s}}.
\end{equation}

% Using \cite[Lemma 2]{chen1998sieve}, we have that $\exists \alpha$ such that $||h||_{\infty} \leq \alpha||h||_2^c$ where $c= \frac{2}{2+d}$. We then get 
% $$|f_{n,h_1} - f_{n,h_2}| \leq \alpha||h_1-h_2||_{2}^c F_n(x, \epsilon)$$

% and hence, $$N_{[]}(\epsilon||F_n||_{P,2}, \mathcal{F}_n, L^2(P)) \leq N\Big(\Big(\frac{\epsilon}{2\alpha}\Big)^{1/c}, \Theta, L^2(P)\Big)$$

% Given the definition of $\Theta$ in \eqref{def function space} and the upper bound on the $\log$ of the metric entropy of smooth functions given in \cite[Example 19.9]{van2000asymptotic}, we obtain:
% $$\log N(\epsilon, \Theta, L^2(P)) \lesssim \Bigg(\frac{1}{\epsilon}\Bigg)^V ,$$ with $V = \frac{4}{d+3}$

% From this and the previous result, we get that:
% $$\log N_{[]}(\epsilon||F_n||_{P,2}, \mathcal{F}_n, L^2(P)) \lesssim \Bigg(\frac{1}{\epsilon}\Bigg)^{\frac{V}{c}}$$

% Since $c = \frac{2}{2+d}$, $\frac{V}{c} = 2\frac{d+2}{d+3} < 2$. 

By combining \eqref{entropy inequality} and \eqref{upper bound metric entropy}, we have that 
$$\int_0^1 \sqrt{\log N_{[]}(\epsilon||F_n||_{P,2}, \mathcal{F}_n, L^2(P))}d\epsilon < \infty,$$ for all $n$ which implies that
$$\int_0^{\delta_n} \sqrt{\log N_{[]}(\epsilon||F_n||_{P,2}, \mathcal{F}_n, L^2(P))}d\epsilon \to 0, \forall \delta_n \downarrow 0.$$

Now, we want to show that the sequence of covariance functions $Pf_{n,s}f_{n,t} - Pf_{n,s}Pf_{n,t}$ converges pointwise on $\Theta\times\Theta$.
$$Pf_{n,s}Pf_{n,t} = \mathbb{E}\Bigg[2s(X)\epsilon -\frac{s(X)^2}{r_n}\Bigg]\mathbb{E}\Bigg[2t(X)\epsilon -\frac{t(X)^2}{r_n}\Bigg] \to 0, $$ as $n \to \infty$
and
$$ Pf_{n,s}f_{n,t} = \mathbb{E} \Bigg[\Bigg(2s(X)\epsilon -\frac{s(X)^2}{r_n}\Bigg) \Bigg(2t(X)\epsilon -\frac{t(X)^2}{r_n}\Bigg) \Bigg] \to 4\sigma^2\mathbb{E}(s(X)t(X)),$$
hence the sequence of covariance functions converges pointwise to $4\sigma^2\mathbb{E}(s(X)t(X))$.

Now, since all hypotheses of Theorem \cite[Theorem 2.11.23]{van1996weak} are satisfied, we have that the sequence of empirical processes $\{\mathbb{G}_n f_{n,h}: h \in K\}$ is asymptotically tight in $l^{\infty}(K)$ and converges in distribution to a tight mean zero Gaussian process $\mathbb{G}$ with covariance $4\sigma^2\mathbb{E}(s(X)t(X))$.

We now need to show that all sample paths $h \to \mathbb{G}(h)$ are upper semi continuous and that the limiting Gaussian process $\{\mathbb{G}_f: f \in \Theta\}$ has a unique maximum at a random function $h^{\star}$.

Let's first note that the space $\Theta$ is $||\cdot||_{L_2(P)}$-compact as shown in \cite{freyberger2015compactness}. 

Concerning continuity of the sample paths of the Gaussian process, as stated in  \cite[Corollary 4.15]{adler1990introduction}, a sufficient condition of continuity of a centered Gaussian process with an index space $\Theta$ totally bounded with respect to the canonical metric $d$ induced by the process, $d(s,t) = \sqrt{\mathbb{E}[(\mathbb{G}_s - \mathbb{G}_t)^2]}$, is that $$\int_0^{\infty} \sqrt{\log N(\epsilon, \Theta, d)} d\epsilon < \infty.$$ 
 
Let's note that in our case the canonical metric induced by the process $d$ is proportional to the regular $||\cdot||_{L^2(P)}$ metric,
\begin{align*}
    d(s,t) &= \sqrt{\mathbb{E}[(\mathbb{G}_s - \mathbb{G}_t)^2]} \\
    &= \sqrt{\mathbb{E}[\mathbb{G}_s^2] + \mathbb{E}[\mathbb{G}_t^2] -2\mathbb{E}[\mathbb{G}_s\mathbb{G}_t]}\\
    &= \sqrt{4\sigma^2\mathbb{E}[s^2] + 4\sigma^2\mathbb{E}[t^2] -8\sigma^2\mathbb{E}[st]}\\
    &= 2\sigma||s-t||_{L^2(P)}.
\end{align*}

Therefore we need to show that
$$\log N(\epsilon, \Theta, L^2(P)) \lesssim \Big(\frac{1}{\epsilon}\Big)^\alpha,$$ with $\alpha < 2$.
This can also be proved thanks to \cite[Corollary 2, Corollary 4]{nickl2007bracketing} and by definition of $\Theta$ in \eqref{def function space}.

This proves the boundedness of the integral of the metric entropy defined under the $d$ metric. The total boundedness of the index space $\Theta$ is induced by its compactness. Hence $\{\mathbb{G}_f: f \in \Theta\}$ is continuous over a compact index space which implies that $\{\mathbb{G}_f: f \in \Theta\}$ reaches its maximum. 
 
Uniqueness of the maximum of a Gaussian process is given in \cite[Lemma 2.6]{kim1990cube} and requires continuity of sample paths, a compact index space and for every $s \neq t$, $\textsf{Var}(\mathbb{G}_s - \mathbb{G}_t) \neq 0$. This is satisfied in our case, indeed, $\textsf{Var}(\mathbb{G}_s - \mathbb{G}_t) = 4\sigma^2\mathbb{E}(s^2+ t^2) - 8\sigma^2\mathbb{E}(st) = 4\sigma^2\mathbb{E}((s-t)^2) \neq 0$.

In order to apply the argmax continuous mapping theorem, as stated in \cite[Theorem 3.2.2]{van1996weak}, we finally need to show that
\begin{equation}
\label{eq:max-criteria}
\mathbb{G}_{n}f_{n,h_n} \geq \sup_{h \in \Theta} \mathbb{G}_n f_{n, h} - o_P(1).
\end{equation}
Given the definition of $f_{n,h}$ and $\mathbb{G}_n f_{n,h}$, \eqref{eq:max-criteria} is equivalent to
\begin{equation}
\label{eq:final-eq}
\frac{r_n}{\sqrt{n}} \sum_{i=1}^n \Big(l_{f_n}(X_i, \epsilon_i) -\mathbb{E}[l_{f_n}(X,\epsilon)] \Big) \geq \sup_{h \in \Theta} \frac{r_n}{\sqrt{n}} \sum_{i=1}^n \Big(l_{f_0 + \frac{h}{r_n}}(X_i, \epsilon_i) - \mathbb{E}[l_{f_0 + \frac{h}{r_n}}(X,\epsilon)] \Big) -o_P(1).
\end{equation}
This condition is satisfied given Condition 4 of Theorem \ref{theorem: NN distribution}.

Finally, from the convergence of the empirical criterion function and the tightness of the random sequence $h_n$ previously shown, one can apply the argmax continuous mapping theorem as stated in \cite[Theorem 3.2.2]{van1996weak}.

%\blue{
%\begin{theorem}[Theorem 3.2.2 of \cite{van1996weak}]
%Let $\mathbb{M}_n$, $\mathbb{M}$ be stochastic processes indexed by a metric space $H$ such that $\mathbb{M}_n \to \mathbb{M}$ in $l^{\infty}(K)$ for every compact $K \subset H$. Suppose that almost all sample paths $h \to \mathbb{M}(h)$ are upper semicontinuous and possess a unique maximum at a point $\hat{h}$, which as a random map in $H$ is tight. If the sequence $\hat{h}_n$ is uniformly tight and satisfy $\mathbb{M}_n(\hat{h}_n) \geq \sup_h \mathbb{M}_n(h) - o_P(1)$, then $\hat{h}_n \to \hat{h}$ in H. 
%\end{theorem}}

This shows the convergence of $h_n$ to the argmax of the Gaussian process $h^{\star}$.
\end{proof}

\begin{proof}[Proof of Theorem \ref{theorem: test statistic distribution}]
The result follows from Theorem \ref{theorem: NN distribution} and the second order functional delta method \cite[Theorem 2]{romisch2005delta}.
We need to compute the Hadamard derivative of our functional of interest $\phi$ as defined in \eqref{functional def}. The second-order Hadamard directional derivative of a mapping between two normed space $\mathbb{D}$ and $\mathbb{F}$, $\phi: \mathbb{D}_0 \subseteq \mathbb{D} \to \mathbb{F}$ at $\theta_0$ in the direction $h$ tangential to $\mathbb{D}_0 $ is defined as
\begin{equation}
\label{Hadamard differentiable}
  \phi''_{\theta_0}(h) = \lim_{t \to 0, n \to \infty} \frac{\phi(\theta_0 + th_n) - \phi(\theta_0) - t\phi'_{\theta_0}(h_n)}{\frac{1}{2}t^2},  
\end{equation}
for every sequence $h_n \to h$ such that $\theta_0 + th_n$ is contained in the domain of $\phi$ for all small $t$.

In our case, the second derivative of $\phi$ at $\theta_0$ under the null hypothesis is equal to 
\begin{equation*}
    \phi^{''}_{\theta_0}(h) = 2\int_{\mathcal{X}}\Bigg(\frac{\partial h(x)}{\partial x_j}\Bigg)^2d\mu(x).
\end{equation*}
This completes the proof.
\end{proof}

\begin{proof}[Proof of Proposition \ref{corollary: empirical test statistic distribution}]
We simplify the notation as $$\hat{\lambda}^n_j = \mathbb{P}_n \eta(f_n),\qquad \lambda^n_j = P \eta(f_n),$$ where $P$ and $\mathbb{P}_n$ denote the expectation and the empirical expectation respectively, and $\eta(f) = \Big(\frac{\partial f}{\partial x_j}\Big)^2$. We have 
\begin{equation*}
%\begin{split}
     r_n^2\Big(\mathbb{P}_n \eta(f_n) - P \eta(f_0)\Big) = \\ \frac{r_n^2}{\sqrt{n}}\mathbb{G}_n\Big(\eta(f_n)-\eta(f_0)\Big) + \frac{r_n^2}{\sqrt{n}}\mathbb{G}_n\Big(\eta(f_0)\Big) + r^2_n\Big(P\eta(f_n) - P\eta(f_0)\Big),   
%\end{split}
\end{equation*}
where $\mathbb{G}_n$ denotes the empirical process operator defined as $\mathbb{G}_n f =\sqrt{n}(\mathbb{P}_n - P f)$. 

%By definition of $r_n$ in \eqref{estimation rate}, it is a $o(n^{1/4})$ so we have that $\frac{r_n^2}{\sqrt{n}} = o(1)$. By the CLT, $\mathbb{G}_n\Big(\eta(\theta_0)\Big)$ converges in distribution to $N\Big(0, \textsf{Var}\Big[(\frac{\partial \theta_0}{\partial x_j})^2(X)\Big]\Big)$ which is equal to 0 under the null. 

Concerning the second term on the right-hand side, given the definition of the null hypothesis $H_0$, we have that $\mathbb{E}[\eta(f_0)] = 0$ which implies that $\frac{\partial f_0}{\partial x_j} = 0$ $P$-almost surely. Thanks to that, we have that $\frac{r_n^2}{\sqrt{n}}\mathbb{G}_n\Big(\eta(f_0)\Big) = 0$ $P$-almost surely.

The distribution of the last term $r^2_n\Big(P\eta(f_n) - P\eta(f_0)\Big)$ is given by Theorem \ref{theorem: test statistic distribution}. 

% For the last term we can use a simpler version of \cite[Theorem 2.3]{van2007empirical}, indeed in our case we do not need uniformity over the index $\theta$ using their notations. Let's call $K$ the compact subset which contains the values of $h^{\star}$.
% $$\mathbb{G}_n(\eta(f_0 + h_0/r_n)-\eta(f_0)) = \frac{1}{r_n^2}\mathbb{G}_n(\eta(h_0))$$ under the null hypothesis.
% By convergence of $\mathbb{G}_n(\eta(h_0))$ to a normal variable with finite variance we have that
% $$|\mathbb{G}_n(\eta(f_0 + h_0/r_n)-\eta(f_0))| \xrightarrow{p} 0$$ for all $h_0 \in K$.

% For $\delta > 0$ and $h_0 \in K$, let's define the sequence of classes of functions
% $$\mathcal{F}_n(h_0,\delta) = \{\eta(f_0 + h/r_n) - \eta(f_0 + h_0/r_n): h \in \Theta, \|h - h_0\| < \delta\}$$
% under the null, we have that:
% $$\eta(f_0 + h/r_n) - \eta(f_0 + h_0/r_n) = \frac{1}{r_n^2}\big(\eta(h)-\eta(h_0)\big)$$
% Because $h \in \Theta$ which is a compact function space as defined in \eqref{def function space}, we can uniformly upper bound its derivative by a constant $M$ and define the envelop function $F(h_0, \delta)$ as
% $$F(h_0, \delta) = \frac{1}{r_n^2}\big(M - \eta(h_0)\big),$$ then
% $$|\mathbb{G}_n F(h_0, \delta)| = \frac{1}{r_n^2}\mathbb{G}_n\big(M - \eta(h_0)\big) \xrightarrow{p} 0$$ by same argument as before.

For the remaining term, we show that the empirical process $\{\mathbb{G}_n f_{n,h} : h \in K\}$ based on the family of functions $\{f_{n,h} = \eta(\theta_0 + h/r_n)-\eta(\theta_0), h \in K\}$ converges to a tight Gaussian process using \cite[Theorem 2.11.23]{van1996weak}. We have

\begin{equation*}
    f_{n,h} = \eta(\theta_0 + h/r_n)-\eta(\theta_0) = \frac{1}{r_n^2}\Big(\frac{\partial h}{\partial x_j}\Big)^2,
\end{equation*}
under the null hypothesis.
By definition of $\Theta$, $\exists M > 0$ such that $\forall h \in \Theta$, $\Big\|\frac{\partial h}{\partial x_j}\Big\|_2 \leq M$.
Because $\frac{\partial h}{\partial x_j}$ is $C^1$ and by \cite[Lemma 2]{chen1998sieve}, we have a uniform boundedness with the infinity norm, i.e. $\exists M > 0$ such that $\forall h \in K$, $\|\frac{\partial h}{\partial x_j}\|_{\infty} \leq M$.

We can then define the envelope function $F_n$ as
\begin{equation*}
    F_n(x) = 2\frac{M^2}{r_n^2},
\end{equation*}
which satisfies
$$\mathbb{E}(F_n^2) = O(1),$$
$$\mathbb{E}(F_n^21_{\{F_n > \eta\sqrt{n}\}}) \to 0,$$ for every $\eta > 0$. 

Since $f_{n,s} - f_{n,t} = \frac{1}{r_n^2}\Bigg[\Big(\frac{\partial s}{\partial x_j}\Big)^2-\Big(\frac{\partial t}{\partial x_j}\Big)^2\Bigg] \to 0$ for all $s,t \in \Theta \times \Theta$ we immediately have that
$$\sup_{\rho(s,t) < \delta_n} \mathbb{E}[(f_{n,s}-f_{n,t})^2] \to 0,$$ for every $\delta_n \to 0$.
Further,
$$|f_{n,s}-f_{n,t}| \leq \frac{1}{r_n^2}\Bigg[\Big(\frac{\partial s}{\partial x_j}\Big)^2-\Big(\frac{\partial t}{\partial x_j}\Big)^2\Bigg]\leq |F_n(x)|\Big|\frac{\partial s}{\partial x_j} - \frac{\partial t}{\partial x_j}\Big|.$$
Hence by the same argument used in the proof of Theorem \ref{theorem: NN distribution}, we obtain
\begin{equation*}
    N_{[]}(\epsilon||F_n||_{P,2}, \mathcal{F}_n, L^2(P)) \leq N\Big(\frac{\epsilon}{2}, \Theta', \|\cdot\|_{\infty}\Big),
\end{equation*}
where $\Theta' = \{\frac{\partial f}{\partial x_j}: f \in \Theta\}$. 

By definition of $\Theta$ and thanks to \cite[Corollary 2 and 4]{nickl2007bracketing} we have that 
$$\log N_{[]}(\epsilon||F_n||_{P,2}, \mathcal{F}_n, L^2(P)) \lesssim \Bigg(\frac{1}{\epsilon}\Bigg)^{\gamma},$$
with $\gamma < 2$.

This last condition ensures that
$$\int_0^{\delta_n} \sqrt{\log N_{[]}(\epsilon||F_n||_{P,2}, \mathcal{F}_n, L^2(P))}d\epsilon \to 0, \forall \delta_n \downarrow 0.$$

By \cite[Theorem 2.11.23]{van1996weak}, we obtain that $\mathbb{G}_n\Big(\eta(\hat{\theta}_n)-\eta(\theta_0)\Big)$ is asymptotically tight.

Then, $\frac{r_n^2}{\sqrt{n}}\mathbb{G}_n\Big(\eta(\hat{\theta}_n)-\eta(\theta_0)\Big)$ converges to 0 in probability. By Slutsky's theorem, we finally get that the limiting distribution of $\hat{\lambda}^n_j$ is the same as that of $\lambda^n_j$.
\end{proof}

\begin{proof}[Proof of Theorem \ref{theorem: new expression limiting distribution}]

The new expression of the limiting distribution is obtained by considering an orthonormal decomposition of the Gaussian process. This is then used to express the argmax of the Gaussian process as the maximizer of a dot product over a ball.

Recall that in our case, the limiting Gaussian process is centered and has covariance $\textsl{Cov}(\mathbb{G}_f,\mathbb{G}_g) = 4\sigma^2\mathbb{E}(fg)$, which is the scaled inner product $\langle f, g \rangle = \mathbb{E}(fg)$ in the Hilbert space $L_2(P)$. If we consider an orthonormal basis $\{\phi_i\}_{i=1}^{\infty}$ of the Hilbert space, then $f = \sum_{i=1}^{\infty} \langle f, \phi_i \rangle \phi_i$, $\forall f \in \Theta$. 

By linearity of the Gaussian process, we have that
\begin{equation}
\label{decomposition gaussian process 1}
    \mathbb{G}_f = \sum_{i=1}^{\infty} \langle f, \phi_i \rangle \mathbb{G}_{\phi_i}.
\end{equation}
The linearity of $\mathbb{G}$ is proved as follows: $\mathbb{G}_{f+ \lambda g} \sim \mathcal{N}(0, 4\sigma^2\mathbb{E}[(f + \lambda g)^2])$ by definition of $\mathbb{G}$ and $\mathbb{G}_{f} + \lambda \mathbb{G}_{g} \sim \mathcal{N}(0, 4\sigma^2\mathbb{E}[f^2]+\lambda^2 4\sigma^2\mathbb{E}[g^2] + \lambda 8\sigma^2\mathbb{E}(fg))$, hence $\mathbb{G}_{f+ \lambda g}$ and $\mathbb{G}_f + \lambda \mathbb{G}_g$ are equal in distribution. 

Because $\{\phi_i\}_{i=1}^{\infty}$ is an orthonormal basis, 
\begin{equation}
\label{distribution gaussian process at basis}
    \mathbb{G}_{\phi_i} \sim \mathcal{N}(0, 4\sigma^2),
\end{equation}
and $\forall i \neq j$, we have $\textsl{Cov}(\mathbb{G}_{\phi_i},\mathbb{G}_{\phi_j}) = 0$. This, combined with the fact that $\mathbb{G}_{\phi_i}$ and $\mathbb{G}_{\phi_j}$ are jointly normal, implies that $\mathbb{G}_{\phi_i}$ is independent of $\mathbb{G}_{\phi_j}$.

Hence, by combining \eqref{decomposition gaussian process 1} and \eqref{distribution gaussian process at basis}, we obtain the following decomposition of the Gaussian process $\forall f \in \Theta$,
\begin{equation}
    \label{decomposition gaussian process 2}
    \sum_{i=1}^{\infty} \langle f, \phi_i \rangle \xi_i,
\end{equation}
where $\{\xi_i\}_{i=1}^{\infty}$ is an infinite sequence of independent normal variables such that $\xi_i \sim \mathcal{N}(0,4\sigma^2)$.

Our objective is to estimate the distribution of the argmax $h^{\star}$ of the previous Gaussian process. The argmax is defined as a random function $\Omega \to \Theta$ such that $\forall \omega \in \Omega$, $h^{\star}_{\omega}$ satisfies $\mathbb{G}_{h^{\star}_{\omega}}(\omega) \geq \mathbb{G}_f(\omega)$, $\forall f \in \Theta$.
The previous expansion of the Gaussian process gives us one way to compute the argmax for each $\omega$. Indeed, for every $\omega \in \Omega$ which corresponds to a sample from the infinite sequence of independent normal variables $\{\xi_i\}_{i=1}^{\infty}$, let's define $\Phi_{\xi}: \Theta \to \mathbb{R}$ as $\Phi_{\xi}(f) = \mathbb{G}_f(\omega) = \sum_{i=1}^{\infty} \langle f, \phi_i \rangle \xi_i$. $\Phi_{\xi}$ is a bounded linear functional on the Hilbert space $L^2(P)$, hence by Riesz representation theorem, $\exists \phi_{\xi} \in L^2(P)$ such that $G_f(\omega) =  \langle f, \phi_{\xi} \rangle_{L_2(P)}$. Hence, finding $f \in \Theta$ that maximizes $\mathbb{G}_f(\omega)$  is equivalent of finding $f \in \Theta$ that maximizes the inner product $\langle f, \phi_{\xi} \rangle_{L^2(P)}$. 

By associating to any $f \in L^2(P)$ its corresponding infinite sequence $\{\langle f,\phi_i\rangle\}_{i \in \mathbb{N}}$, one can identify $L^2(P)$ with the $l^2$ space of infinite sequences with bounded 2-norm. Hence, \eqref{stable basis} shows that the Sobolev norm $\|\cdot\|_{\mathbb{H}^k(P)}$ is a weighted version of the $l^2$ norm with weights $d_i^2$. 

Thanks to the previous observation, finding $f \in \Theta$ that maximizes $\mathbb{G}_f(\omega)$ can be re-expressed as solving the following optimization problem,
\begin{equation*}
   \max_{\{\langle f, \phi_i \rangle\}_{i \in \mathbb{N}}} \langle f, \phi_{\xi} \rangle_{l^2} = \sum_{i=0}^{\infty} \langle f, \phi_i \rangle \langle \phi_{\xi}, \phi_i \rangle,
\end{equation*}
subject to $\|Df\|^2_{l^2} =  \sum_{i=0}^{\infty} d_i^2\langle f, \phi_i \rangle^2 \leq B$.

By defining $g = Df$, the optimization problem above can be reformulated as
$$\max \langle g, D^{-1}\phi_{\xi} \rangle_{l^2},$$ subject to $\|g\|_{l^2}  \leq B$ whose solution by Cauchy-Schwartz inequality obviously is
$$g^{\star} = \frac{B}{\|D^{-1}\phi_{\xi}\|_{L^2(P)}}D^{-1}\phi_{\xi}.$$

Because $f = D^{-1} g$ and $\phi_{\xi} = \sum_{i=0}^{\infty} \xi_i \phi_i$, the argmax can be computed explicitly as
\begin{equation*}
    h^{\star}_{\omega} = \sum_{i = 0}^{\infty} \frac{B}{\|D^{-1}\phi_{\xi}\|}\frac{\xi_i}{d_i^2} \phi_i,
\end{equation*}
where $\|D^{-1}\phi_{\xi}\|^2 = 4\sigma^2\sum_{i=0}^{\infty} \frac{\chi_i^2}{d_i^2}$ and $\{\chi_i^2\}_{i \in \mathbb{N}}$ are i.i.d. samples from a chi-square distribution. 

Given the argmax $h^{\star}$ and because of \eqref{basis function derivative}, we can compute the associated statistic $$\mathbb{E}_X\Bigg(\frac{\partial h^{\star}(X)}{\partial x_j}^2\Bigg) = \Bigg\|\frac{\partial h^{\star}(X)}{\partial x_j}\Bigg\|^2_{L^2(P)} = \frac{B^2}{\sum_{i=0}^{\infty} \frac{\chi_i^2}{d_i^2}}\sum_{i = 0}^{\infty} \frac{\alpha_{i,j}^2}{d_i^4} \chi_{i}^2.$$
This completes the proof.
\end{proof}

\section{Variables for House Price Valuation}\label{table: variables descriptions}

\small
\begin{longtable}{| p{0.46\textwidth} | p{0.47\textwidth} |}
\hline
\textbf{Variable Name} & \textbf{Variable Description} \\ \hline
\multicolumn{2}{|c|}{\textbf{House Variables}}\\ \hline
Bedrooms &	Number of bedrooms\\ \hline
Full\_Baths	& Number of bathrooms\\ \hline
Last\_Sale\_Amount &	Previous sale price\\ \hline
SqFt	& Square footage\\ \hline
Stories\_Number	& Number of stories\\ \hline
Time\_Since\_Prior\_Sale &	Time since prior sale\\ \hline
Year\_Built & Year of construction \\ \hline

\multicolumn{2}{|c|}{\textbf{Mortgage Variables}}\\ \hline

Last\_Mortgage\_Amount & Amount of the mortgage of the last sale \\ \hline
Last\_Mortgage\_Interest\_Rate & Interest rate of the mortgage\\ \hline

\multicolumn{2}{|c|}{\textbf{Tax Variables}}\\ \hline

Tax\_Amount & The Total Tax amount provided by the county or local taxing/assessment authority. \\ \hline
Tax\_Assd\_Improvement\_Value & The Assessed Improvement Values as provided by the county or local taxing/assessment authority. \\ \hline
Tax\_Assd\_Land\_Value & The Assessed Land Values as provided by the county or local taxing/assessment authority.
 \\ \hline
Tax\_Assd\_Total\_Value & The Total Assessed Value of the Parcel's Land and Improvement values as provided by the county or local taxing/assessment authority.
 \\ \hline
Tax\_Improvement\_Value\_Calc & The "IMPROVEMENT" Value closest to current market value used for assessment by county or local taxing authorities.
 \\ \hline
Tax\_Land\_Square\_Footage & Total land mass in Square Feet.
 \\ \hline
Tax\_Land\_Value\_Calc & The "LAND" Value closest to current market value used for assessment by county or local taxing authorities.  \\ \hline
Tax\_Total\_Value\_Calc & The "TOTAL" (i.e., Land + Improvement) Value closest to current market value used for assessment by county or local taxing authorities. 
\\ \hline

\multicolumn{2}{|c|}{\textbf{Foreclosure Variables}}\\ \hline

N\_Foreclosures & How many foreclosures occurred in that given month in that given CBSA \\ \hline
N\_Outstanding\_Foreclosures\_SAF & Number of properties that have already been foreclosed upon but have not yet been sold to a party other than the original lender for the given CBSA in the given month. \\ \hline
N\_Outstanding\_Foreclosures\_NonInst\_SAF & Same as above but for non-institutional / non-corporate owner SAF. \\ \hline 
Foreclosures\_To\_Loans\_Ratio & Ratio of N\_Outstanding\_Foreclosures\_SAF and N\_Outstanding\_Loans \\ \hline

\multicolumn{2}{|c|}{\textbf{Other Variables}}\\ \hline

Sale\_Month\_No & Month of the sale \\ \hline
N\_Past\_Sales & Number of past sales \\ \hline

\multicolumn{2}{|c|}{\textbf{Mortgage Rate Variables}}\\ \hline

Freddie\_30yrFRM\_5/1ARM\_Spread\_PriorSale & Spread between Freddie Mac 30-year fixed mortgage rate and 5/1 hybrid amortizing adjustable-rate mortgage (ARM) \\ \hline
Freddie\_5/1\_ARM\_Rate\_PriorSale & Freddie Mac 5/1 hybrid amortizing ARM initial coupon rate \\ \hline
Freddie\_5/1\_ARM\_margin\_PriorSale & Freddie Mac 5/1 hybrid amortizing ARM margin \\ \hline
Freddie\_5/1\_feespts\_PriorSale & Freddie Mac 5/1 hybrid amortizing ARM fees and points \\ \hline
Freddie\_Regional\_NC\_1yrARM\_PriorSale & Freddie Mac 1-Year ARM North Central \\ \hline
Freddie\_Regional\_NE\_1yrARM\_PriorSale & Freddie Mac 1-Year ARM North East \\ \hline
Freddie\_Regional\_SE\_1yrARM\_PriorSale & Freddie Mac 1-Year ARM South East \\ \hline
Freddie\_Regional\_SW\_1yrARM\_PriorSale & Freddie Mac 1-Year ARM South West \\ \hline
Freddie\_Regional\_W\_1yrARM\_PriorSale & Freddie Mac 1-Year ARM West \\ \hline
Freddie\_US\_15yr\_FRM\_PriorSale & Freddie Mac 15-Year fixed mortgage rate\\ \hline
Freddie\_US\_15yr\_feespts\_PriorSale & Freddie Mac 15-Year fixed mortgage rate fees and points \\ \hline
Freddie\_US\_1yrARM\_Fees\_PriorSale & Freddie Mac 1-Year ARM fees and points \\ \hline
Freddie\_US\_1yrARM\_Margin\_PriorSale & Freddie Mac 1-Year ARM margin \\ \hline
Freddie\_US\_1yrARM\_PriorSale & Freddie Mac 1-Year ARM \\ \hline
Freddie\_US\_30yr\_FRM\_PriorSale & Freddie Mac 30-Year fixed mortgage rate \\ \hline
Freddie\_US\_30yr\_feespts\_PriorSale & Freddie Mac 30-Year fixed mortgage rate fees and points \\ \hline

\multicolumn{2}{|c|}{\textbf{Zip-code Variables}}\\ \hline

N\_ShortSales & Number of short sales that occurred within the same zipcode as the property within the past 12 months. \\ \hline
N\_PrePayments & Number of loans that have been prepaid \\ \hline
N\_Outstanding\_Loans & How many loans are currently outstanding \\ \hline
N\_PIF & Number of loans that were paid in full by reaching their term maturity \\ \hline
N\_Originations	& Number of loans originated at same location and same time period of sale\\ \hline
Median\_Time\_to\_SAF & Median time between the original foreclosure month and the SAF month of all SAF that occurred in the given month. \\ \hline
Median\_Time\_to\_NonInst\_SAF & Same as above but for non-institutional / non-corporate owner SAF. \\ \hline
Median\_Time\_to\_SAF\_l12 & Same as above except that computation rolls over the 12 preceding months.\\ \hline
Median\_Time\_to\_NonInst\_SAF\_l12 & Same as above except that computation rolls over the 12 preceding months. \\ \hline
Median\_time\_since\_foreclosure\_SAF & Median time between the current month and when the foreclosure originally occurred of all properties that have already been foreclosed upon but have not yet been sold.\\ \hline
Median\_time\_since\_foreclosure\_NonInst\_SAF & Same as above but for non-institutional / non-corporate owner sales after foreclosure. \\ \hline
Median\_SAF\_PriceChange & Median difference between price of the SAF and last recorded sale price of the property for all SAF that occurred in current month. \\ \hline
Median\_NonInst\_PriceChange & Same as above but for non-institutional / non-corporate owner SAF. \\ \hline
Median\_SAF\_PriceChange\_l12 & Same as above except that computation rolls over the 12 preceding months. \\ \hline
Median\_NonInst\_PriceChange\_l12 & Same as above except that computation rolls over the 12 preceding months. \\ \hline
Median\_SAF\_PriceChange\_pct & Same as above except that the price change is expressed as a percentage, i.e. the price change divided by the previous sale price. \\ \hline
Median\_NonInst\_PriceChange\_pct & Same as above except that the price change is expressed as a percentage. \\ \hline
Median\_SAF\_PriceChange\_pct\_l12 & Same as above except that the price change is expressed as a percentage. \\ \hline
Median\_NonInst\_PriceChange\_pct\_l12 & Same as above except that the price change is expressed as a percentage. \\ \hline
Median\_SAF\_PriceChange\_Model & Median difference between price of the SAF and a modeled sale price of the property (last sale price $*[1+$ zip code level House Price Index precent change since last sale$]$) for all SAF that occurred in current month.\\ \hline
Median\_NonInst\_PriceChange\_Model & Same as above but for non-institutional / non-corporate owner SAF. \\ \hline
Median\_SAF\_PriceChange\_Model\_l12 & Same as above except that computation rolls over the 12 preceding months. \\ \hline
Median\_NonInst\_PriceChange\_Model\_l12 & Same as above except that computation rolls over the 12 preceding months. \\ \hline
Median\_SAF\_PriceChange\_Model\_pct & Same as above except that the price change is expressed as a percentage, i.e. the price change divided by the previous sale price. \\ \hline
Median\_SAF\_PriceChange\_Model\_pct\_l12 & Same as above except that the price change is expressed as a percentage. \\ \hline
Median\_NonInst\_SAF PriceChange\_Model\_pct & Same as above except that the price change is expressed as a percentage. \\ \hline 
Median\_NonInst\_SAF PriceChange\_Model\_pct\_l12 & Same as above except that the price change is expressed as a percentage. \\ \hline
Time\_to\_NonInst\_to\_Outstanding\_Loans & Ratio of Median\_Time\_to\_NonInst\_SAF\_l12 variable over N\_Outstanding\_Loans \\ \hline
%\caption{Feature variables available in house price data set}
\end{longtable}


\newpage

\vskip 0.2in
\bibliography{biblio_NN.bib}

\end{document}